\documentclass[final]{siamltex}
\usepackage{geometry}
\usepackage{graphicx}
\usepackage{amssymb}
\usepackage{amsmath}
\usepackage{epstopdf}
\usepackage{setspace}
\usepackage{url}
\usepackage{color}
\usepackage{cancel}
\usepackage{color}

\usepackage{lscape}

\usepackage[english]{babel}

\usepackage{latexsym,amssymb,amsfonts,graphicx}
\usepackage{verbatim}
\usepackage{mathrsfs}
\usepackage{bm}
\usepackage{color}
\usepackage{epsfig}
\usepackage{bbm}

\usepackage{url}

\usepackage{bm}

\usepackage{lscape}

\newcommand{\indicator}[1]{\mathbbm{1}_{\left[ {#1} \right] }}
\newcommand{\E}{\mathbb E}

\newcommand{\eps}{\epsilon}

\newcommand{\tr}{\top}

\newtheorem{condition}{Condition}[section]
\DeclareMathOperator*{\argmax}{arg\,max}

\newcommand{\LF}{\mathcal{L}^{F}}
\newcommand{\LS}{\mathcal{L}^{S}}

\date{\today}

\begin{document}
\bibliographystyle{apalike}

\title{Dimension Reduction in Statistical Estimation of Partially Observed Multiscale Processes}
\author{Andrew Papanicolaou\footnotemark[2]\and Konstantinos Spiliopoulos\footnotemark[3]}
\footnotetext[2]{Department of Finance and Risk Engineering, NYU Tandon School of Engineering, 6 MetroTech Center, Brooklyn NY 11201 {\em ap1345@nyu.edu}.}
\footnotetext[3]{Department of Mathematics \& Statistics, Boston University, 111 Cummington Street, Boston MA 02215, \em{kspiliop@math.bu.edu}.
Work partially supported  by NSF grant DMS-1312124 and by NSF CAREER award-DMS 1550918 .}

\date{\today}

\maketitle
\begin{abstract}
We consider partially observed multiscale diffusion models that are specified up to an unknown vector parameter. We establish for a very general class of test functions that the filter of the original model converges to a filter of reduced dimension. Then, this result is used to justify statistical estimation for the unknown parameters of interest based on the model of reduced dimension but using the original available data. This allows to learn the unknown parameters of interest while working in lower dimensions, as opposed to working with the original high dimensional system. Simulation studies support and illustrate the theoretical results.
\end{abstract}

{\bf Keywords. } data assimilation, filtering, parameter estimation, homogenization, multiscale diffusions, dimension reduction. \\
{\bf Subject classifications. } 93E10 93E11 93C70 62M07 62M86

\date{}
\section{Introduction}

This paper considers statistical inference for a general filtering problem with multiple time scales. On a probability space $(\Omega,(\mathcal F_t)_{t\leq T},\mathbb P)$ with $T<\infty$, for positive integers $m,k,n$ we consider the $(m+k+n)$-dimensional process $(X^{\delta},U^{\delta}, Y^{\delta})=\{(X^{\delta}_{t},U^{\delta}_{t},Y^{\delta}_{t})\in\mathbb R^m\times\mathbb R^k\times\mathbb R^n, 0\leq t\leq T\}\in C([0,T];\mathbb R^m\times\mathbb R^k\times\mathbb R^n)$, which satisfies a system of stochastic differential equations (SDEs) with $0<\delta\ll 1$,
\begin{eqnarray}
dY^{\delta}_{t}&=&h_{\theta}\left(X^{\delta}_{t},U^{\delta}_{t}\right)dt +  dW_{t}\nonumber\hspace{4.4cm}\hbox{(observed)}\\
dU^{\delta}_{t}&=& g_{\theta}\left(X^{\delta}_{t},U^{\delta}_{t}\right)dt + \tau_{\theta}\left(X^{\delta}_{t},U^{\delta}_{t}\right)dV_{t}\nonumber\hspace{2.9cm}\hbox{(hidden)}\\
dX^{\delta}_{t}&=& \frac{1}{\delta}b_{\theta}\left(X^{\delta}_{t},U^{\delta}_{t}\right)dt + \frac{1}{\sqrt{\delta}}\sigma_{\theta}\left(X^{\delta}_{t},U^{\delta}_{t}\right)dB_{t}\label{Eq:Model3}\hspace{2cm}\hbox{(hidden)}\ ,
\end{eqnarray}
where $(W_t)_{t\leq T}$, $(B_t)_{t\leq T}$ and $(V_t)_{t\leq T}$ are (unobserved) independent Wiener processes in $\mathbb{R}^{n}$, $\mathbb R^m$ and $\mathbb R^k$, respectively.

One possible interpretation in the context of financial applications, is that the vector $Y_t$ is part of a continuous stream of financial data, with the hidden processes $X_t^\delta$ and $U_t^\delta$ being factors in a stochastic model, see for example \cite{FouqueEtall2011}. A different possible interpretation is that $Y^{\delta}_{t}$ is the signal of brain activity during a seizure, as in \cite{Jisra2014} where  multiscale dynamical systems are exploited as models for seizure dynamics. Other applications include multiscale modeling in oceanography and climate modeling, see for example \cite{Majda2008}.

Initially, the model in \eqref{Eq:Model3} is left unspecified, and it is with the arrival of data that we can learn the parameter value $\theta$. We let $\mathcal Y_t^\delta$ denote the $\sigma$-algebra generated by the data, $\mathcal Y_t^\delta=\sigma \{(Y_s^\delta)_{s\leq t}\}$, and a primary goal is to compute the maximum likelihood estimator (MLE). Given $\mathcal Y_t^\delta$, we denote the log-likelihood function as $\rho_t^\delta(\theta)$, which we maximize to obtain the estimator
\[\hat\theta_t^\delta = \arg\max_\theta \rho_t^\delta(\theta)\ .\]

In financial applications for example and in particular in  high-frequency trading, market and limit orders need to be executed at exact moments in time with accuracy that is of the order of nano seconds. Hence, an estimate such as the MLE may be very accurate, but it may also require computing time that is too long for the purposes of trading. In general, it is well known that standard Monte-Carlo particle filters can be quite slow even without the effect of multiple scales.  However, in the small-$\delta$ limit there is a significant reduction in dimension because the process $X^\delta$ gets `averaged out' by the ergodic theory, and this allows filtering for $U_t^\delta$ and inference for $\theta$ to have a faster algorithm.

The novel contribution of this paper is at two levels. At the first level, we prove that the nonlinear filter for (\ref{Eq:Model3}) converges, for any parameter value of $\theta$ and under the measure parameterized by the true parameter value, with the limit being a filter of reduced dimension. We prove this result for test functions that can depend on both the slow and fast unknown components, $X$ and $U$; the test functions may be unbounded but we do impose moments bounds. This result extends previous works of \cite{ImkellerSriPerkowskiYeong2012} and \cite{PapanicolaouSpiliopoulos2014}. At the second level, we use the filters convergence result to prove that the MLE based on the model of reduced dimension produces both consistent and asymptotically normal estimators, and we identify the limiting variance of the estimator. We apply our methodology to examples where the reduction procedure allows to use classical Kalman-filter type of methods for a problem that is highly nonlinear without the small-$\delta$ asymptotics. Our examples demonstrate how statistical inference for a highly nonlinear problem is reduced to statistical inference for a problem that is both linear and of reduced dimension, and hence is significantly less demanding from a computational standpoint. Computations of filters in problems with nonlinearity in the reduced system are considerably more complicated to implement and usually require particle filers; we refer the reader to \cite{GivonStinisWeare2009,Papavasiliou2007} for some related results in the multiscale setting. Another implementation issue is model learning/estimation, which is sometimes done with an expectation maximization (EM) algorithm (see \cite{elliott93}). The results in this paper are directly applicable to learning of parametric models if the multiscale framework applies; the model is parameterized by $\theta$ and an MLE can be obtained using the reduced model as opposed to the original model.

The rest of the paper is organized as follows. In Section \ref{S:ProblemFormulation} we formulate the problem of interest in specific terms and present some preliminary well-known results that are useful throughout.
In Section \ref{S:FilteringEquations} we present the filtering equations of the original problem and we derive the filters associated to the problem of reduced dimensions. In Section \ref{S:FilterConvergence} we present our first main result on filter convergence which is the main justification for doing inference using the original available data $\mathcal Y_t^\delta$ but based on the model of reduced dimension. In Section \ref{S:MLE_reducedEstimator} we present our main results on parameter estimation. In particular, we prove that under appropriate conditions, the MLE for the problem of reduced dimension is both consistent and asymptotically normal as $T\rightarrow\infty$. In Section \ref{S:Example} we consider a few examples to illustrate and supplement the theoretical results for which classical Kalman filter techniques can be used. Simulation data are presented to support the theoretical claims.

\section{Problem formulation and preliminary results}\label{S:ProblemFormulation}

Let $\theta\in\Theta\subset \mathbb{R}^{d}$ be an unknown parameter of interest which is to be estimated from data generated by (\ref{Eq:Model3}). Assuming that we observe only the $Y^{\delta}$ process, we develop a consistent and asymptotically normal MLE for $\theta$. We consider the case $\delta\ll 1$, in which case $X^{\delta}$ is the fast component and $U^{\delta}$ is the slow component. We reduce the high dimensionality of the problem by looking at the $\delta\downarrow 0$ limit and exploiting the statistical properties of the MLE based on the reduced likelihood.

For the model given in \eqref{Eq:Model3}, we shall write $\LF$ and $\LS$ for the infinitesimal generators of the fast process, $X^{\delta}$, and the process $U^{\delta}$, respectively. Namely,
\begin{align}
\LF_{\theta,u} f(x)&=b_\theta(x,u)\cdot D_{x}f(x)+\frac{1}{2}\textrm{tr}\left[\sigma_\theta(x,u)\sigma_\theta^{\tr}(x,u)D^{2}_{x}f(x) \right]\nonumber\\
\LS_{\theta,x} f(u)&=g_\theta(x,u)\cdot D_{u}f(u)+\frac{1}{2}\textrm{tr}\left[\tau_\theta(x,u)\tau_\theta^{\tr}(x,u)D^{2}_{u}f(u) \right]
\end{align}

Next we pose the main condition of this paper on the growth and regularity  of the coefficients of (\ref{Eq:Model3}). Such conditions guarantee
that (\ref{Eq:Model3}) has a well defined strong solution,
that the fast component $X_t^\delta$ is ergodic and that the slow component $Y_t^\delta$ has a well-defined homogenization limit as $\delta\downarrow 0$ in
the appropriate sense. Assumptions to guarantee these properties are given by \cite{PardouxVeretennikov1} for homogenization and Chapter 3 of \cite{bainCrisan} for filtering, and are stated as follows:

\begin{condition}\label{A:Assumption1}
\begin{enumerate}
\item{For ergodicity purposes we shall assume the recurrence condition
\[
\lim_{|x|\rightarrow\infty}\sup_{\theta\in\Theta}\sup_{u\in\mathbb{R}^{k}} b_{\theta}(x,u)\cdot x=-\infty.
\]

Under this assumption the Lyapunov-type condition for existence of an invariant measure associated to the fast dynamics $X$ of Hasminskii \cite{Hasminskii} is satisfied.
}
\item{To guarantee uniqueness of the  invariant measure for $X$, we assume that $\sigma_{\theta}(x,u)\sigma_{\theta}^{T}(x,u)$
is  non-degenerate uniformly in $\theta$, i.e., there exists $\lambda(\theta)>0$ such that for all $(x,u)\in\mathcal{X}\times\mathcal{U}$
\[
\left|\xi\sigma_{\theta}(x,u)\right|^{2}\geq \lambda(\theta)|\xi|^{2},\quad\textrm{ for all }(\theta,\xi)\in\Theta\times\mathbb{R}^{m}.
\]}
\item{
The functions $b_{\theta}(x,u)$ and  $\sigma_{\theta}(x,u)$  are  $C^{2+\alpha,2}_{b}(\mathbb{R}^{m}\times\mathbb{R}^{k})$ with $\alpha\in(0,1)$, uniformly
in $\theta\in\Theta$. Namely, uniformly
in $\theta\in\Theta$, they have two bounded derivatives in $x$ and $u$, with all partial derivatives being H\"{o}lder continuous,
with exponent $\alpha$, with respect to $x$, uniformly in $u$.}
\item{  We assume that $g_{\theta}\in C(\mathbb{R}^{m}\times \mathbb{R}^{k})$ and that there exist $K$ and $q$ such that
    \[
    |g(x,u)|\leq K(1+|x|)\left(1+|u|^{q}\right).
    \]}
\item{ For every $N>0$ there exists a constant $C(N)$ such that for all $u_{1},u_{2}\in\mathbb{R}^{m}$ and $|x|\leq N$, the diffusion matrix $\tau$ satisfies
\[
\left|\tau(x,u_{1})-\tau(x,u_{2})\right|\leq C(N)|u_{1}-u_{2}|.
\]
Moreover, there exists $K>0$ and  $q>0$ such that
\[
|\tau(x,u)|\leq K (1+|u|^{1/2})(1+|x|^{q}).
\]
}
\item{$h_{\theta}\in C(\mathbb{R}^{m}\times\mathbb{R}^{k})$ is bounded and globally Lipschitz in $(x,u)\in \mathbb{R}^{m}\times\mathbb{R}^{k}$
uniformly in $\theta\in\Theta$.} 
\item{The functions $h_{\theta},b_{\theta},\sigma_{\theta}, g_{\theta},\tau_{\theta}$ are Lipschitz continuous in $\theta\in\Theta$ and $\Theta \subset\mathbb{R}^{d}$ is bounded open set.}
\end{enumerate}
\end{condition}

Under this assumption and non-degeneracy of the diffusion coefficient $\sigma_\theta(x,u)\sigma_\theta^\tr(x,u)$, for fixed with $U_{t}=u$ the process $X^{1}$ has a unique invariant measure which we shall denote by $\mu_{\theta}(dx;u)$. For a given function $f\in\mathcal{L}^{2}(\mu_{\theta})$, define its averaged version as
\[
\bar{f}_{\theta}(u)=\int_{\mathbb{R}^{m}}f(x,u)\mu_{\theta}(dx;u)
\]
It is a well known result that $Y^{\delta}_{\cdot}$ converges in distribution in $C([0,T];\mathbb{R}^n)$ to the process $\overline{Y}_{\cdot}$
 (e.g. \cite{BLP,PardouxVeretennikov2}), where
\begin{equation}
\overline{Y}_{t}=\int_{0}^{t}\bar{h}_{\theta}(\bar{U}_{s}) ds+  W_{t}.\label{Eq:LimitingModel}
\end{equation}
and where
\begin{equation}
d\bar{U}_{t}= \bar{g}_{\theta}\left(\bar{U}_{t}\right)dt + \bar{\tau}_{0,\theta}\left(\bar{U}_{t}\right)dV_{t}, \quad \text{ with }\bar{\tau}_{0,\theta}(u)=\left(\overline{\tau_{\theta}(\cdot,u)\tau_{\theta}^\tr(\cdot,u)}\right)^{1/2}\ .
\end{equation}
Actually, due to the fact that the observation process $Y^{\delta}_{t}$ has constant diffusion, Condition \ref{A:Assumption1} and the ergodic theorem guarantee that a stronger result holds for any $\theta\in\Theta$, i.e., for every $\varepsilon>0$
\begin{equation}
\mathbb{P}_\theta\left(\sup_{0\leq t\leq T}\left|Y^{\delta}_{t}-\overline{Y}_{t}\right|\geq \varepsilon \right)\rightarrow 0, \textrm{ as }\delta\downarrow 0\qquad\forall\theta\in\Theta.\label{Eq:MeanSquareConvergence}
\end{equation}


\section{Filtering equations}\label{S:FilteringEquations}
 The data is contained in the filtration
\[\mathcal Y_t^\delta \doteq\mathcal F_t^{Y^\delta} = \sigma \{(Y_s^\delta)_{s\leq t}\}\ ,\]
which is right continuous and $\mathcal Y_0^\delta$ contains all $\mathbb{P}-$negligible sets. For any $\theta\in\Theta$, we define  a new measure $\mathbb{P}_{\theta}^{*}$ on $(\Omega,\mathcal{F})$ via the relationship
\begin{equation}
\frac{d\mathbb{P}_{\theta}}{d\mathbb{P}^{*}_{\theta}}\doteq Z_T^{\delta,\theta}=
\exp\left( \int_{0}^{T}h_{\theta}(X^{\delta}_{s},U^{\delta}_{s})dY^{\delta}_{s}-\frac{1}{2}\int_{0}^{T}\left|h_{\theta}(X^{\delta}_{s},U^{\delta}_{s})\right|^{2}ds   \right)\ .\label{Eq:LogLikelihoodPreLimit}
\end{equation}
Under the proper assumptions $Z_T^{\delta,\theta}$ is an exponential martingale and thus
the probability measures $\mathbb{P}_{\theta}$ and $\mathbb{P}^{*}_{\theta}$ are absolutely continuous with respect to each other,
and the distribution of $(U^{\delta},X^{\delta})$ is the same under both $\mathbb{P}_{\theta}$ and $\mathbb{P}_{\theta}^{*}$.
Furthermore,  the process $Y^\delta$ is a $\mathbb{P}_{\theta}^*$-Brownian motion and is independent of $(U^{\delta},X^{\delta})$.

Next, for  $f:\mathcal X\times\mathcal U\rightarrow\mathbb R$ such that $\mathbb{E}_{\theta}^{*}\left[ f(X^{\delta}_{t},U^{\delta}_{t})\right]<\infty$, we define the measure valued process $\phi^{\delta,\theta}_{t}$ acting on $f$ as
\begin{equation}\label{eq:expFormulation}
\phi^{\delta,\theta}_{t}[f] \doteq\mathbb{E}_{\theta}^{*}\left[ Z_t^{\delta,\theta}f(X^{\delta}_{t},U^{\delta}_{t}) \Big |\mathcal{Y}_{t}^\delta\right]\ ,
\end{equation}
a process which, for $f\in C^{2}_{c}(\mathcal{X}\times\mathcal{U})$ is well-known to be the unique solution (see \cite{rozovsky1992}) to the following equation:
\begin{eqnarray}
d\phi^{\delta,\theta}_{t}[f]&=&\left(\frac{1}{\delta}\phi^{\delta,\theta}_{t}[\LF_{\theta}f] +\phi^{\delta,\theta}_{t}[\LS_{\theta}f]   \right)dt
+ \phi^{\delta,\theta}_{t}[h_{\theta} f] dY_s^\delta,\quad \mathbb{P}_{\theta}^{*} \textrm{-a.s. }\ ,\nonumber\\
   \phi_0^{\theta}[f]&=& \mathbb E_{\theta}f(X_0^\delta,U_0^\delta)\ .\label{Eq:Zakai}
 \end{eqnarray}
Equation \eqref{Eq:Zakai} is the Zakai equation for nonlinear filtering.
Furthermore, $\phi_t^{\delta,\theta}$ is actually an unnormalized probability measure which yields the normalized posterior expectations via the Kalianpour-Striebel
formula,
\begin{equation}
\label{Eq:KSformula}
\pi^{\delta,\theta}_{t}[f]\doteq\E_{\theta}\left[f(X_t^\delta,U_t^\delta)\Big|\mathcal Y_t^\delta\right]=\frac{\phi_t^{\delta,\theta}[f]}{\phi_t^{\delta,\theta}[1]}\quad \mathbb{P}_{\theta},\mathbb{P}_{\theta}^{*} \textrm{-a.s.}\ .
\end{equation}
If $f(x,u)=h_{\theta}(x,u)$ then we have the innovations process,
\[\nu_t^{\delta,\theta}\doteq Y_t^\delta-y_0 - \int_0^t\pi^{\delta,\theta}_{s}[h]ds\qquad\forall t\in[0,T]\ .\]
The process $\nu_t^{\delta,\theta}$ is a $\mathbb P_{\theta}$-Brownian motion under the filtration generated by the observed $Y^{\delta}$ process, but will only be observable as Brownian motion if $\theta$ is equal to the true parameter value. For a suitable test function $f:\mathcal X\times \mathcal{U}\rightarrow \mathbb R$, the innovation process is used in the nonlinear Kushner-Stratonovich equation to describe the evolution of $\pi_t^{\delta,\theta}[f]$,
\begin{equation}
\label{Eq:Kushner}
d\pi_t^{\delta,\theta}[f] = \left(\frac 1\delta \pi^{\delta,\theta}_{t}[\LF_{\theta} f]+\pi^{\delta,\theta}_{t}[\LS_{\theta} f] \right)dt+ \left(\pi^{\delta,\theta}_{t}[f h_{\theta}]-\pi^{\delta,\theta}_{t}[f]\pi^{\delta,\theta}_{t}[h_{\theta}]\right)d\nu_t^{\delta,\theta}\quad \mathbb{P}_{\theta}\textrm{-a.s.} \end{equation}

Let us next consider the filtering equations of the approximating problem that is of reduced dimension. Consider the `averaged' exponentials
\begin{align}
 \bar{Z}_t^{\delta,\theta}&\doteq \exp\left( \int_{0}^{t}\bar{h}_{\theta}(\bar{U}_{s}) dY_s^\delta-\frac{1}{2}\int_{0}^{t}\left|\bar{h}_{\theta}(\bar{U}_{s})\right|^{2}ds   \right)\ , \label{Eq:DefinitionsZ}\\
 \bar{Z}_t^{\theta}&\doteq \exp\left( \int_{0}^{t}\bar{h}_{\theta}(\bar{U}_{s}) d\bar{Y}_s-\frac{1}{2}\int_{0}^{t}\left|\bar{h}_{\theta}(\bar{U}_{s})\right|^{2}ds   \right)\ .
\end{align}
For $f\in\mathcal{C}^{2}_{c}(\mathcal{U})$, we define new posterior measures $\bar{\phi}^{\delta,\theta}_{t}[f]$ and $\bar{\phi}^{\theta}_{t}[f]$ which satisfy the stochastic evolution equations
\begin{eqnarray}
\label{Eq:avgZakaiDelta}
d\bar{\phi}^{\delta,\theta}_{t}[f]&=&\bar{\phi}^{\delta,\theta}_{t}[\overline{L}_{\theta}^{S}f]dt+\bar{\phi}^{\delta,\theta}_t[\bar{h}_{\theta} f]  dY^{\delta}_t, \quad \bar{\phi}^{\delta,\theta}_{0}[f]=\mathbb{E}_{\theta}f(\bar{U}_{0})\ ,\\
\label{Eq:avgZakaiLimit}
d\bar{\phi}^{\theta}_{t}[f]&=&\bar{\phi}^{\theta}_{t}[\overline{L}_{\theta}^{S}f]dt+\bar{\phi}^{\theta}_t[\bar{h}_{\theta}f]  d\bar{Y}_t, \quad \bar{\phi}^{\theta}_{0}[f]=\mathbb{E}_{\theta}f(\bar{U}_{0})\ .
 \end{eqnarray}
It is straightforward to verify with It\^o's lemma that the `average' Zakai equations  \eqref{Eq:avgZakaiDelta} and \eqref{Eq:avgZakaiLimit} have solutions
\begin{eqnarray}
\label{Eq:solAvgZakai}
\bar{\phi}^{\delta,\theta}_{t}[f]&=& \mathbb E_{\theta}^*\left[ f_{\theta}(\bar{U}_{t}) \bar{Z}_t^{\delta,\theta}\Big|\mathcal Y_t^\delta\right] \ ,\\
\label{Eq:solAvgZakaiLimit}
\bar{\phi}^{\theta}_{t}[f]&=&  \mathbb E_{\theta}^*\left[ f_{\theta}(\bar{U}_{t}) \bar{Z}_t^{\theta}\Big|\mathcal{\bar{Y}}_t\right],\
\end{eqnarray}
where $\mathcal{\bar{Y}}_t=\mathcal F_t^{\bar{Y}} = \sigma \{(\bar{Y}_s)_{s\leq t}\}$ is a right continuous $\sigma-$algebra and $\mathcal{\bar{Y}}_0$ contains all $\mathbb{P}-$negligible sets. We  define the averaged filters as follows:
$\bar{\pi}^{\delta,\theta}_{t}[f]=\frac{\bar{\phi}^{\delta,\theta}_{t}[f]}{\bar{\phi}^{\delta,\theta}_{t}[1]} $
and $\bar{\pi}^{\theta}_{t}[f]=\frac{\bar{\phi}^{\theta}_{t}[f]}{\bar{\phi}^{\theta}_{t}[1]}$. It is easy to see that $\bar{\pi}^{\theta}_{t}[f]=\E_{\theta}\left[f(\bar{U}_t)\Big|\mathcal{\bar{Y}}_t\right]$.

We conclude this section by mentioning that, under $\mathbb{P}_{\theta}$, equation (\ref{Eq:KSformula}) defines a measure-valued process $P^{\delta,\theta}_{\cdot}$, the conditional distribution, by the formula
\[
\left(P^{\delta,\theta}_{t},f\right)=\pi^{\delta,\theta}_{t}[f]=\E_{\theta}\left[f(X^{\delta}_t,U^{\delta}_t)\Big|\mathcal{Y}^{\delta}_t\right]\ .
\]
Similarly, we define the probability measure-valued processes $\bar{P}^{\delta,\theta}_{\cdot}$ and $\bar{P}^{\theta}_{\cdot}$ by
\[
\left(\bar{P}^{\delta,\theta}_{t},f\right)=\bar{\pi}^{\delta,\theta}_{t}[f]\quad\text{ and }\quad \left(\bar{P}^{\theta}_{t},f\right)=\bar{\pi}^{\theta}_{t}[f]
\]
The measure-valued process $\bar{P}^{\theta}$ especially, will become handy in proving the consistency and asymptotic normality of the MLE based on the reduced estimator.
\section{Convergence of the filters}\label{S:FilterConvergence}
Consider $\eta>0$ and define the following class of test functions
\begin{equation*}
\mathcal{A}_{\eta}^{\theta}=\left\{f\in C(\mathcal X\times\mathcal{U})\cap L^{2}(\mathcal{X}\times\mathcal{U};\mu_{\theta}):
\sup_{\delta\in(0,1)}\sup_{t\in[0,T]}\mathbb{E}_{\theta}\left|f(X^{\delta}_{t},U^{\delta}_{t})\right|^{2+\eta}<\infty \right\}.
\end{equation*}
Then, we have the following result which is a generalization of the results of \cite{ImkellerSriPerkowskiYeong2012} and \cite{PapanicolaouSpiliopoulos2014}.
\begin{theorem}\label{T:FilterConvergence1}
Assume Conditions \ref{A:Assumption1}. For any $\alpha,\theta\in\Theta$, the following hold uniformly in $t\in[0,T]$
\begin{enumerate}
\item{Let $f\in C_{b}(\mathcal{X}\times\mathcal{U})$. Then,  for every $\varepsilon>0$
\[
\lim_{\delta\downarrow 0}\mathbb{P}_{\alpha}\left( \left|\phi_t^{\delta,\theta}[f]-\bar\phi_t^{\delta,\theta}[\bar{f}]\right| \geq \varepsilon \right)=0
\]
}
\item{Assume that there is $\eta>0$ such that $f\in\mathcal{A}_{\eta}^{\theta}$. Then, $\pi_t^{\delta,\theta}[f]$ converges in $\mathbb{P}_{\alpha}$-mean-square to $\bar \pi_t^{\theta,\alpha}[f]$, i.e.,
\[\lim_{\delta\downarrow 0}\E_{\alpha}\left(\pi_t^{\delta,\theta}[f]-\bar \pi_t^{\delta, \theta}[\bar{f}]\right)^2=0\ . \]
Additionally, we also have
\[\lim_{\delta\downarrow 0}\E_{\alpha}\left(\bar\pi_t^{\delta,\theta}[\bar{f}]-\bar \pi_t^{\theta}[\bar{f}]\right)^2=0\ .\]}
\end{enumerate}
\end{theorem}

For $f\in C^{4}_{b}(\mathcal{U})$ and for the case $\alpha=\theta$ this is proven in \cite{ImkellerSriPerkowskiYeong2012}. Also,  for the case $h_{\theta}(x,u)=h_{\theta}(x), b_{\theta}(x,u)=b_{\theta}(x)$ and $\sigma_{\theta}(x,u)=\sigma_{\theta}(x)$, i.e., when the model does not include the hidden slow process $U$,  Theorem \ref{T:FilterConvergence1} is proven in \cite{PapanicolaouSpiliopoulos2014}.
 Hence, Theorem \ref{T:FilterConvergence1} extends the results of \cite{ImkellerSriPerkowskiYeong2012,PapanicolaouSpiliopoulos2014} to the case $f\in\mathcal{A}_{\eta}^{\theta}$ and under parameter mismatch. We emphasize here that since we are interested in the parameter estimation, we are naturally interested in making sure that the filters converge for \textit{any} parameter value under the measure parameterized by the true parameter value. The proof of this theorem is in Appendix \ref{A:FilterConvergence1}.

\section{On statistical inference based on reduced likelihood function}\label{S:MLE_reducedEstimator}

Theorem \ref{T:FilterConvergence1} suggests that for parameter estimation, we can approximate the conditional log-likelihood
\begin{equation}
\rho_T^\delta(\theta)=\log \phi^{\delta,\theta}_{T}[1]=\log\mathbb{E}_{\theta}^{*}\left[ Z_T^{\delta,\theta} \Big |\mathcal{Y}^{\delta}_T\right]\label{Eq:LogLikelihood}
\end{equation}
by the `reduced' log-likelihood
\begin{equation}
\bar{\rho}_T^\delta(\theta)=\log \bar{\phi}^{\delta,\theta}_{T}[1]=\log \mathbb{E}_{\theta}^{*}\left[ \bar{Z}_T^{\delta,\theta} \Big |\mathcal{Y}^{\delta}_T\right]\label{Eq:IntermediateLikelihood}
\end{equation}

Note that by Lemma 3.29 in \cite{bainCrisan} we have
\begin{equation}
\bar{\rho}_T^\delta(\theta)=\log\left( \bar{\phi}_T^{\delta,\theta}[1]\right)
= \int_0^T \bar{\pi}_s^{\delta,\theta}[\bar{h}_{\theta}]dY^{\delta}_s -\frac 12\int_0^T\left| \bar{\pi}_s^{\delta,\theta}[\bar{h}_{\theta}]\right|^2ds\
\end{equation}

The following condition is about regularity  of $\bar{h}_{\theta}(u)$ as a function of $\theta\in\Theta$.
\begin{condition}\label{A:Assumption4}
There are constants $C>0$,  $p\geq 1$ and $q>1$, such that for any $\theta_{1},\theta_{2}\in\Theta$,
\[
\sup_{u\in\mathcal{U}}|\bar{h}_{\theta_{1}}(u)-\bar{h}_{\theta_{2}}(u)|^{p}\leq C|\theta_{1}-\theta_{2}|^{q}.
\]
i.e., $\theta\mapsto \bar{h}_{\theta}(u)$ is H\"{o}lder continuous uniformly in $u\in\mathcal{U}$.
\end{condition}

Allowing arbitrary initial conditions, the set  of processes $\left(X^{\delta},U^{\delta}, P^{\delta,\theta}, \bar{P}^{\delta,\theta}\right)$,  $\left(X^{\delta},U^{\delta},  \bar{P}^{\delta,\theta}\right)$ and
$\left(\bar{U}, \bar{P}^{\theta}\right)$ are Markov-Feller processes in $C([0,\infty);\mathbb{R}^{m+k})\times C^2([0,\infty),\mathcal{P}(\mathbb{R}^{k}))$, $C([0,\infty);\mathbb{R}^{m+k})\times C([0,\infty),\mathcal{P}(\mathbb{R}^{k}))$ and $C([0,\infty);\mathbb{R}^{k})\times C([0,\infty),\mathcal{P}(\mathbb{R}^{k}))$. Let $\Pi^{\delta,\theta}(t,\cdot)$, $\bar{\Pi}^{\delta,\theta}(t,\cdot)$ and $\bar{\Pi}^{\theta}(t,\cdot)$ be the corresponding transition functions.  In order to have enough ergodicity of the averaged problem as $T\rightarrow\infty$ we make the following assumption, see for example \cite{Kushner}. The examples that will be considered in Section \ref{S:Example} satisfy Condition \ref{A:Assumption4b}.
\begin{condition}\label{A:Assumption4b}
There is a unique invariant measure $\bar{\Pi}^{\theta}_{0}(\cdot)$ for the transition function $\bar{\Pi}^{\theta}(t,\cdot)$. In addition the set $\{\bar{U}_{t}, t<\infty\}$ is tight and $\lim_{N\rightarrow\infty}\sup_{t}\bar{P}^{\theta}_{t}\left(\{|u|>N\}\right)=0$.
\end{condition}

We also need the following identifiability condition.
\begin{condition}\label{A:AssumptionIdentifiability}
We assume that for any $\theta\in\Theta$, any  $\epsilon>0$, $T<\infty$ and for every $t\in(0,T]$,  one has
\[
\inf_{\alpha\in\Theta}\inf_{|\theta-\alpha|>\epsilon}\mathbb{E}_{\alpha}\left|\bar{\pi}^{\theta}_{t}[\bar{h}_{\theta}]-\bar{\pi}^{\alpha}_{t}[\bar{h}_{\alpha}]\right|^{2}>0.
\]
\end{condition}

Let us define
\begin{equation}
\bar{\theta}^{\delta}_{T}\doteq\argmax_{\theta\in\Theta}\bar{\rho}_T^\delta(\theta).\label{Eq:ThetaMLE}
\end{equation}

Since we have assumed that $\Theta$ is bounded, we get that $\bar{\theta}^{\delta}_{T}\in\textrm{cl}(\Theta)$ with probability $1$. We then have the following theorem:
\begin{theorem}\label{T:ConsisitencyReducedLikelihood}
Assume Conditions \ref{A:Assumption1}, \ref{A:Assumption4}, \ref{A:Assumption4b} and \ref{A:AssumptionIdentifiability}. The maximum likelihood estimator based on (\ref{Eq:IntermediateLikelihood}) is strongly consistent as
first $\delta\downarrow 0$ and then $T\rightarrow\infty$, i.e., for any $\varepsilon>0$
\[
\lim_{T\rightarrow\infty}\lim_{\delta\downarrow 0}\mathbb{P}_{\alpha}\left(\left|\bar{\theta}^{\delta}_{T}-\alpha\right|>\varepsilon\right)=0.
\]
\end{theorem}

\begin{proof}
Under $\mathbb{P}_{\alpha}$, we recall that the innovations process
\[ \nu_t\doteq Y_t^\delta-y_0 - \int_0^t\pi^{\delta,\alpha}_{s}[h_{\alpha}]ds\qquad\forall t\in[0,T]\ \]
is a $\mathbb P_{\alpha}$-Brownian motion under the filtration generated from the observed $Y^{\delta}$ process. Hence, under $\mathbb{P}_{\alpha}$  we have
\begin{align*}
\bar{\rho}_T^\delta(\theta,\alpha)&=\bar{\rho}_T^\delta(\theta)\\
&= \int_0^T \bar{\pi}_s^{\delta,\theta}[\bar{h}_{\theta}]dY^{\delta}_s -\frac 12\int_0^t\left| \bar{\pi}_s^{\delta,\theta}[\bar{h}_{\theta}]\right|^2ds\\
&= \int_0^T \left[\bar{\pi}_s^{\delta,\theta}[\bar{h}_{\theta}]\cdot\pi^{\delta,\alpha}_{s}[h_{\alpha}]-\frac 12\left| \bar{\pi}_s^{\delta,\theta}[\bar{h}_{\theta}]\right|^2\right]ds
+\int_0^T \bar{\pi}_s^{\delta,\theta}[\bar{h}_{\theta}] d\nu_{s}\ ,
\end{align*}
where $\bar{\pi}_s^{\delta,\theta}[\bar{h}_{\theta}]\cdot\pi^{\delta,\alpha}_{s}[h_{\alpha}]$ denotes the inner product over the $Y$ coordinates in $\mathbb R^n$. Then, we have
\begin{align}
\mathbb{E}_{\alpha}\left|\bar{\rho}_T^\delta(\theta_{1},\alpha)-\bar{\rho}_T^\delta(\theta_{2},\alpha) \right|^{p}
&\leq C
\sup_{u\in\mathcal{U}}|\bar{h}_{\theta_{1}}(u)-\bar{h}_{\theta_{2}}(u)|^{p}
\left(1+\mathbb{E}_{\alpha}\int_{0}^{T}|h_{\theta}(X^{\delta}_{s},U^{\delta}_{s})|^{2}ds \right),\nonumber\\
&\leq C |\theta_{1}-\theta_{2}|^{q}\ ,\label{Eq:RegularityReducedLikelihood}
\end{align}
where we used Condition \ref{A:Assumption4}. The constant $C$ might change from line to line, but we do not indicate this in the notation. By Theorem \ref{T:FilterConvergence1}, we have
that
\begin{equation}
\lim_{\delta\rightarrow 0}\int_{0}^{T}\mathbb{E}_{\alpha} \left|\pi^{\delta,\alpha}_{s}[h_{\alpha}]-\bar{\pi}_s^{\alpha}[\bar{h}_{\alpha}]\right|^{2}ds=0\ .\label{Eq:Convergence1}
\end{equation}
Theorem \ref{T:FilterConvergence1},  (\ref{Eq:RegularityReducedLikelihood})  and (\ref{Eq:Convergence1}) imply by Theorem 12.3 in \cite{Billingsley}
that we have convergence in distribution of the process $\bar{\rho}^{\delta}_{T}(\cdot,\alpha)$ to that of $\bar{\rho}_{T}(\cdot,\alpha)$ in the uniform metric, where
\begin{align}
\bar{\rho}_T(\theta,\alpha)&=
\int_0^T \left(\bar{\pi}_s^{\theta}[\bar{h}_{\theta}]\cdot\bar{\pi}^{\alpha}_{s}[\bar{h}_{\alpha}]-\frac 12\left| \bar{\pi}_s^{\theta}[\bar{h}_{\theta}]\right|^2\right)ds
+\int_0^T \bar{\pi}_s^{\theta}[\bar{h}_{\theta}] d\nu_{s}\nonumber\\
&=
-\frac{1}{2}\int_0^T \left|\bar{\pi}_s^{\theta}[\bar{h}_{\theta}]-\bar{\pi}_s^{\alpha}[\bar{h}_{\alpha}]\right|^{2}ds+
\frac{1}{2}\int_0^T \left|\bar{\pi}_s^{\alpha}[\bar{h}_{\alpha}]\right|^{2}ds+\int_0^T \bar{\pi}_s^{\theta}[\bar{h}_{\theta}] d\nu_{s}
\end{align}
Hence, similarly to the proof of Proposition 1.32 in page 61 of \cite{Kutoyants}, we have
\begin{align}
\lim_{\delta\downarrow 0}\mathbb{P}_{\alpha}\left(\left|\bar{\theta}^{\delta}_{T}-\alpha\right|>\varepsilon\right)&=
\lim_{\delta\downarrow 0}\mathbb{P}_{\alpha}\left(\sup_{\left|\theta-\alpha\right|>\varepsilon}\frac{1}{T}\bar{\rho}_{T}^\delta(\theta,\alpha)>\sup_{\left|\theta-\alpha\right|\leq\varepsilon}\frac{1}{T}\bar{\rho}_{T}^\delta(\theta,\alpha)\right)\nonumber\\
&=
\mathbb{P}_{\alpha}\left(\sup_{\left|\theta-\alpha\right|>\varepsilon}\frac{1}{T}\bar{\rho}_{T}(\theta,\alpha)>\sup_{\left|\theta-\alpha\right|\leq\varepsilon}\frac{1}{T}\bar{\rho}_{T}(\theta,\alpha)\right)\ .\nonumber
\end{align}
At the same time for any $\eta>0$ we have
\begin{align}
\mathbb{P}_{\alpha}\left(\left|\frac{1}{T}\int_0^T \bar{\pi}_s^{\theta}[\bar{h}_{\theta}]d\nu_{s}\right|>\eta\right)&\leq\frac{1}{\eta^{2}T^{2}}\mathbb{E}_{\alpha}\int_0^T \left|\bar{\pi}_s^{\theta}[\bar{h}_{\theta}]\right|^{2}ds\nonumber\\
&\leq\frac{1}{\eta^{2}T}C_{0}\nonumber
\end{align}
for some constant $C_{0}$ that does not depend on $T$. Thus, we have in $\mathbb{P}_{\alpha}-$probability that
\[
\lim_{T\rightarrow\infty}\left[\frac{1}{T}\int_0^T \bar{\pi}_s^{\theta}[\bar{h}_{\theta}]d\nu_{s}\right]=0\ .
\]
Let us define
\begin{align*}
\bar{\bar{\rho}}_T(\theta,\alpha)&=
-\frac{1}{2}\int_0^T \left|\bar{\pi}_s^{\theta}[\bar{h}_{\theta}]-\bar{\pi}_s^{\alpha}[\bar{h}_{\alpha}]\right|^{2}ds+
\frac{1}{2}\int_0^T \left|\bar{\pi}_s^{\alpha}[\bar{h}_{\alpha}]\right|^{2}ds\ ,
\end{align*}
and recalling Condition \ref{A:Assumption4b} we get that there is $\bar{\bar{\rho}}_{\infty}(\theta,\alpha)$ such that
\[
\lim_{T\rightarrow\infty}\mathbb{E}\left|\frac{1}{T}\bar{\bar{\rho}}_T(\theta,\alpha)-\bar{\bar{\rho}}_{\infty}(\theta,\alpha)\right|=0\ .
\]
Hence, we obtain
\begin{align}
\lim_{T\rightarrow\infty}\lim_{\delta\downarrow 0}\mathbb{P}_{\alpha}\left(\left|\bar{\theta}^{\delta}_{T}-\alpha\right|>\varepsilon\right)&=
\lim_{T\rightarrow\infty}\mathbb{P}_{\alpha}\left(\sup_{\left|\theta-\alpha\right|>\varepsilon}\frac{1}{T}\bar{\rho}_{T}(\theta,\alpha)>\sup_{\left|\theta-\alpha\right|\leq\varepsilon}\frac{1}{T}\bar{\rho}_{T}(\theta,\alpha)\right)\nonumber\\
&\leq
\lim_{T\rightarrow\infty}\mathbb{P}_{\alpha}\left(\sup_{\left|\theta-\alpha\right|>\varepsilon}\frac{1}{T}\bar{\bar{\rho}}_{T}(\theta,\alpha)>\sup_{\left|\theta-\alpha\right|\leq\varepsilon}\frac{1}{T}\bar{\bar{\rho}}_{T}(\theta,\alpha)\right)\nonumber\\
&=
\mathbb{P}_{\alpha}\left(\sup_{\left|\theta-\alpha\right|>\varepsilon}\bar{\bar{\rho}}_{\infty}(\theta,\alpha)>\sup_{\left|\theta-\alpha\right|\leq\varepsilon}\bar{\bar{\rho}}_{\infty}(\theta,\alpha)\right)\nonumber\\
&=0\ ,
\end{align}
where the last computation used the identifiability Condition \ref{A:AssumptionIdentifiability}. With this, we conclude the proof of the theorem.
\end{proof}

Let us study next asymptotic normality of the maximum likelihood estimator that is based on the reduced likelihood function. We have that the maximizer will be solution to the equation $\frac{\partial}{\partial\theta}\bar{\rho}^{\delta}_{T}=0$ for $\theta\in\Theta$. Thus, the maximizer $\tilde{\theta}=\tilde{\theta}^{\delta}_{T}$ of that equation will satisfy the equation
\begin{equation}
\int_{0}^{T}\nabla_{\theta}\bar\pi_{s}^{\delta,\tilde{\theta}}[\bar{h}_{\tilde{\theta}}]\left(dY_{s}^{\delta}-\bar\pi_{s}^{\delta,\tilde{\theta}}[\bar{h}_{\tilde\theta}]ds\right)=0\ .\label{Eq:EquationForReducedMLE}
\end{equation}
 We mention here that (\ref{Eq:ThetaMLE}) and (\ref{Eq:EquationForReducedMLE}) are not equivalent; (\ref{Eq:ThetaMLE}) contains all local minima and local maxima of $\bar{\rho}^{\delta}_{T}(\theta)$ which may be more than one. Also equation \eqref{Eq:EquationForReducedMLE} may not even have a solution in $\Theta$ with positive probability.
For example, letting $\tilde{\theta}^{\delta}_{T}$ be a solution to (\ref{Eq:EquationForReducedMLE}) and assuming $\theta\in(\theta_\ell,\theta_u)$, then
\[
\bar{\theta}^{\delta}_{T}=\tilde{\theta}^{\delta}_{T}\indicator{\{\tilde{\theta}^{\delta}_{T}\in(\theta_{\ell},\theta_{u})\}}+\theta_{\ell}\indicator{\tilde{\theta}^{\delta}_{T}\leq \theta_{\ell}\}}+\theta_{u}\indicator{\{\tilde{\theta}^{\delta}_{T}\geq \theta_{u}}.
\]

Next we study asymptotic normality of the MLE corresponding to the reduced log-likelihood. We make the following assumption.
\begin{condition}\label{A:ExtraConditionForCLT}
There exists a strictly positive definite matrix $I(\alpha)$ such that we have in $L^{1}$ under the measure $\mathbb{P}_{\alpha}$
\begin{equation}
\label{eq:fisherInfo}
I(\alpha)=\lim_{T\rightarrow\infty}\frac{1}{T}\int_{0}^{T}\nabla_{\alpha}\bar\pi_{s}^{\alpha}[\bar{h}_{\alpha}]^\tr
\cdot\nabla_{\alpha}\bar\pi_{s}^{\alpha}[\bar{h}_{\alpha}] ds \ .
\end{equation}
\end{condition}
It is clear that in the case $\bar{h}_{\theta}(u)=\bar{h}_{\theta}$ that Condition \ref{A:ExtraConditionForCLT} is satisfied with constant matrix $I(\alpha)=\nabla_{\alpha}\bar{h}_{\alpha}\cdot\nabla_{\alpha}\bar{h}_{\alpha}$. In Section \ref{S:Example} a more involved example will be examined where things can be also computed explicitly in closed form.
Actually, $I(\alpha)$ is nothing else but the Fisher information matrix. By Theorem \ref{T:ConsisitencyReducedLikelihood}, based on smoothness of
$\nabla_{\theta}\bar\pi_{s}^{\delta,\theta}[\bar{h}_{\theta}]$ as a function of $\theta$ and under Condition \ref{A:ExtraConditionForCLT}, asymptotic normality of the MLE corresponding to the reduced log-likelihood holds. To be precise, we have the following theorem.
\begin{theorem}\label{T:CLTReducedLikelihood}
Assume Conditions \ref{A:Assumption1}, \ref{A:Assumption4}, \ref{A:Assumption4b}, \ref{A:AssumptionIdentifiability}, \ref{A:ExtraConditionForCLT} and that $\nabla_{\theta}\bar\pi_{s}^{\delta,\theta}[\bar{h}_{\theta}]$ is almost surely continuous as a function of $\theta$.  The maximum likelihood estimator based on (\ref{Eq:IntermediateLikelihood}) is asymptotically normal
under $\mathbb{P}_{\alpha}$, i.e.
\begin{equation}
\label{E:CLTReducedLikelihood}
\sqrt{T}\left(\bar{\theta}^{\delta}_{T}-\alpha\right)\Rightarrow N\left(0,I^{-1}(\alpha)\right)\qquad\hbox{first as $\delta\downarrow 0$ and then as $T\rightarrow\infty$,}
\end{equation}
where $I(\alpha)$ is Fisher information given by \eqref{eq:fisherInfo}.
\end{theorem}
\begin{proof} 
We write $\dot{\bar\pi}_{s}^{\delta,\theta}[\bar{h}_{\theta}]=\nabla_{\theta}\bar\pi_{s}^{\delta,\theta}[\bar{h}_{\theta}]$ for notational convenience.
Based on (\ref{Eq:EquationForReducedMLE}) and for $\theta=\bar{\theta}=\bar{\theta}^{\delta}_{T}$, we write for some $\alpha^{*}$ such that $|\alpha^{*}-\alpha|\leq |\bar{\theta}-\alpha|$
\begin{align}
0&=\int_{0}^{T}\dot{\bar\pi}_{s}^{\delta,\bar{\theta}}[\bar{h}_{\bar\theta}]\left(dY_{s}^{\delta}-\bar\pi_{s}^{\delta,\bar{\theta}}[\bar{h}_{\bar\theta}]ds\right)\nonumber\\
&=\int_{0}^{T}\dot{\bar\pi}_{s}^{\delta,\bar{\theta}}[\bar{h}_{\bar\theta}]\left(dY_{s}^{\delta}-\bar\pi_{s}^{\delta,\alpha}[\bar{h}_{\alpha}]ds-(\bar{\theta}-\alpha)\dot{\bar\pi}_{s}^{\delta,\alpha^{*}}[\bar{h}_{\alpha^{*}}]ds\right)\nonumber\\
&=\int_{0}^{T}\dot{\bar\pi}_{s}^{\delta,\bar{\theta}}[\bar{h}_{\bar\theta}]dY_{s}^{\delta}
-\int_{0}^{T}\dot{\bar\pi}_{s}^{\delta,\bar{\theta}}[\bar{h}_{\bar\theta}]\cdot\bar\pi_{s}^{\delta,\alpha}[\bar{h}_{\alpha}]ds
-(\bar{\theta}-\alpha)\int_{0}^{T}\dot{\bar\pi}_{s}^{\delta,\bar{\theta}}[\bar{h}_{\bar\theta}]\cdot\dot{\bar\pi}_{s}^{\delta,\alpha^{*}}[\bar{h}_{\alpha^{*}}]ds\ .\nonumber
\end{align}
After some term rearrangement, we obtain
\begin{align*}
\sqrt{T}\left(\bar{\theta}^{\delta}_{T}-\alpha\right)&=\left(\frac{1}{T}\int_{0}^{T}\dot{\bar\pi}_{s}^{\delta,\bar{\theta}}[\bar{h}_{\bar\theta}]\cdot\dot{\bar\pi}_{s}^{\delta,\alpha^{*}}[\bar{h}_{\alpha^{*}}]ds\right)^{-1}
\frac{1}{\sqrt{T}}\int_{0}^{T}\dot{\bar\pi}_{s}^{\delta,\bar{\theta}}[\bar{h}_{\bar\theta}] d\nu_{s}\nonumber\\
&+\left(\frac{1}{T}\int_{0}^{T}\dot{\bar\pi}_{s}^{\delta,\bar{\theta}}[\bar{h}_{\bar\theta}]\cdot\dot{\bar\pi}_{s}^{\delta,\alpha^{*}}[\bar{h}_{\alpha^{*}}]ds\right)^{-1}
\frac{1}{\sqrt{T}}\int_{0}^{T}\dot{\bar\pi}_{s}^{\delta,\bar{\theta}}[\bar{h}_{\bar\theta}]\cdot\left(\pi^{\delta,\alpha}_{s}[h_{\alpha}]-\bar\pi_{s}^{\delta,\alpha}[\bar{h}_{\alpha}]\right)ds\ ,
\end{align*}
and by taking $\delta\rightarrow0$, ergodcity and Theorem \ref{T:FilterConvergence1} guarantee that
\begin{align}
\lim_{\delta\rightarrow 0}\mathbb{E}_{\alpha}\int_{0}^{T}\left|\pi^{\delta,\alpha}_{s}[h_{\alpha}]-\bar\pi_{s}^{\delta,\alpha}[\bar{h}_{\alpha}]\right|^{2}ds&=0\ .
\end{align}
The latter statement and Condition \ref{A:ExtraConditionForCLT}  guarantee us that in $\mathbb{P}_{\alpha}$-probability as first $\delta\downarrow 0$ and then $T\rightarrow\infty$
\begin{align}
&\left(\frac{1}{T}\int_{0}^{T}\dot{\bar\pi}_{s}^{\delta,\bar{\theta}}[\bar{h}_{\bar\theta}]\cdot\dot{\bar\pi}_{s}^{\delta,\alpha^{*}}[\bar{h}_{\alpha^{*}}]ds\right)^{-1}
\frac{1}{\sqrt{T}}\int_{0}^{T}\dot{\bar\pi}_{s}^{\delta,\bar{\theta}}[\bar{h}_{\bar\theta}]\cdot\left(\pi_{s}^{\delta,\alpha}[h_{\alpha}]-\bar\pi_{s}^{\delta,\alpha}[\bar{h}_{\alpha}]\right)ds\rightarrow 0\ .\label{Eq:LimitCLT1}
\end{align}
For notational convenience, let us define the random matrix
\[
f^{\delta}_{T}(\theta_{1},\theta_{2})=\frac{1}{T}\int_{0}^{T}\dot{\bar\pi}_{s}^{\delta,\theta_{1}}[\bar{h}_{\theta_{1}}]\cdot\dot{\bar\pi}_{s}^{\delta,\theta_{2}}[\bar{h}_{\theta_{2}}]ds\ .
\]
Since under $\mathbb{P}_{\alpha}$, the innovations process
\[ \nu_t\doteq Y_t^\delta-y_0 - \int_0^t\pi^{\delta,\alpha}_{s}[h_{\alpha}]ds\qquad\forall t\in[0,T]\ \]
is  a $\mathbb P_{\alpha}$-Brownian motion, for the stochastic integral we notice:
 \begin{align}
&\left(\frac{1}{T}\int_{0}^{T}\dot{\bar\pi}_{s}^{\delta,\bar{\theta}}[\bar{h}_{\bar\theta}]\cdot\dot{\bar\pi}_{s}^{\delta,\alpha^{*}}[\bar{h}_{\alpha^{*}}]ds\right)^{-1}
\frac{1}{\sqrt{T}}\int_{0}^{T}\dot{\bar\pi}_{s}^{\delta,\bar{\theta}}[\bar{h}_{\bar\theta}] d\nu_{s}=\left(f^{\delta}_{T}(\bar{\theta},\alpha^{*})\right)^{-1}
\frac{1}{\sqrt{T}}\int_{0}^{T}\dot{\bar\pi}_{s}^{\delta,\bar{\theta}}[\bar{h}_{\bar\theta}]d\nu_{s}\nonumber
 \end{align}

Since $|\alpha^{*}-\alpha|\leq |\bar{\theta}^{\delta}_{T}-\alpha|$,  Theorem \ref{T:ConsisitencyReducedLikelihood} implies that
\[
\lim_{T\rightarrow\infty}\lim_{\delta\downarrow 0}\mathbb{P}_{\alpha}\left(\left|\alpha^*-\alpha\right|>\varepsilon\right)\leq \lim_{T\rightarrow\infty}\lim_{\delta\downarrow 0}\mathbb{P}_{\alpha}\left(\left|\bar\theta_T^\delta-\alpha\right|>\varepsilon\right)=0\qquad\hbox{for any $\varepsilon>0$,}
\]
and hence by the almost sure continuity of $\dot{\bar\pi}_{s}^{\delta,\theta}[\bar{h}_{\theta}]$ as a function of $\theta$ and  Condition \ref{A:ExtraConditionForCLT}, we obtain that in $\mathbb{P}_{\alpha}$ probability as $\delta\downarrow 0$ and then $T\rightarrow\infty$
\[
f^{\delta}_{T}\left(\bar{\theta},\alpha^{*}\right)\rightarrow I(\alpha)\ .
\]
Then, Proposition 1.21 in \cite{Kutoyants} and Slutsky's theorem imply that in
distribution, first as $\delta\downarrow 0$ and then as $T\rightarrow\infty$,
\begin{align}
\left(\frac{1}{T}\int_{0}^{T}\dot{\bar\pi}_{s}^{\delta,\bar{\theta}}[\bar{h}_{\bar\theta}]\cdot\dot{\bar\pi}_{s}^{\delta,\alpha^{*}}[\bar{h}_{\alpha^{*}}]ds\right)^{-1}
\frac{1}{\sqrt{T}} \int_{0}^{T}\dot{\bar\pi}_{s}^{\delta,\bar{\theta}}[\bar{h}_{\bar\theta}]d\nu_{s}\Rightarrow N\left(0,I^{-1}(\alpha)\right).\label{Eq:LimitCLT2}
\end{align}

Limits (\ref{Eq:LimitCLT1}) and (\ref{Eq:LimitCLT2}) imply then the statement of the theorem by Slutsky's theorem on the combined expression.
\end{proof}

\section{Examples}\label{S:Example}
In this section we consider numerically several examples in order to illustrate the results of this paper. Even though the theory of this paper has been developed under the assumption that $h_{\theta}$ is bounded, the numerical results of this section indicate that there is some degree of flexibility to this assumption and that the results should be broader applicable.  Let $\theta$ be a finite-dimensional parameter and consider the system of equations
\begin{eqnarray}
dY^{\delta}_{t}&=&a(\theta)\lambda\left(X^{\delta}_{t}\right)U^{\delta}_{t}dt + \Sigma dW_{t}\nonumber\hspace{3.6cm}\hbox{(observed)}\\
dU^{\delta}_{t}&=& -\beta(\theta)q\left(X^{\delta}_{t}\right)U^{\delta}_{t}dt + \gamma(\theta)dV_{t}\nonumber\hspace{2.9cm}\hbox{(hidden)}\\
dX^{\delta}_{t}&=& \frac{1}{\delta}b_{\theta}\left(X^{\delta}_{t}\right)dt + \frac{1}{\sqrt{\delta}}\sigma_{\theta}\left(X^{\delta}_{t}\right)dB_{t}\label{Eq:ModelExamplePrelimit}\hspace{2.7cm}\hbox{(hidden)}
\end{eqnarray}
where $\beta(\theta)>0$ and $a(\theta)\neq 0, \gamma(\theta)\neq 0$, $\Sigma\neq 0$, and $Y,U,X$ take values in $\mathbb{R}^{1}$. Without loss of generality and for presentation purposes the theory of the paper was developed for $\Sigma=I$, but as we shall see here the same results hold with $\Sigma\neq I$ as long as $\Sigma$ is non-degenerate.

Let us further assume that we know the invariant measure for $X$ and that it is given by $\mu(dx)$. Then, we know that the limit of (\ref{Eq:ModelExamplePrelimit}) in probability as $\delta\downarrow 0$ is
\begin{eqnarray}
d\bar{Y}_{t}&=a(\theta)\bar{\lambda}\bar{U}_{t}dt+ \Sigma dW_{t}\nonumber\\
d\bar{U}_{t}&=-\beta(\theta)\bar{q}\bar{U}_{t}dt+\gamma(\theta)dV_{t}\label{Eq:ModelExampleLimit}\ .
\end{eqnarray}

It is relatively easy to see that the diffusion coefficient $\Sigma$ can be viewed as a scaling factor. Under $\mathbb{P}_{\alpha}$, the process $\{\nu_{t}, t\in[0,T]\}$ defined by the equation
\[ \nu_t=\frac1\Sigma\left(Y_t^\delta-y_0 - \int_0^t\pi^{\delta,\alpha}_{s}[h_{\alpha}]ds\right)\qquad\forall t\in[0,T]\ \]
is a $\mathbb P_{\alpha}$-Brownian motion under the filtration generated from the observed $Y^{\delta}$ process. The maximizer satisfies
\[\frac{1}{\Sigma^2} \int_{0}^{T}\nabla_{\theta}\bar\pi_{s}^{\delta,\tilde{\theta}}[\bar{h}_{\tilde{\theta}}]\left(dY_{s}^{\delta}-\bar\pi_{s}^{\delta,\tilde{\theta}}[\bar{h}_{\tilde\theta}]ds\right)=0\ ,\]
and the Fisher information turns out to be
\[I(\alpha)=\lim_{T\rightarrow\infty}\frac{1}{T\Sigma^{2}}\int_{0}^{T}\nabla_{\alpha}\bar\pi_{s}^{\alpha}[\bar{h}_{\alpha}]^\tr
\cdot\nabla_{\alpha}\bar\pi_{s}^{\alpha}[\bar{h}_{\alpha}] ds \ .\]

The limiting system (\ref{Eq:ModelExampleLimit}) uses the well-known Kalman-Bucy filter. The inference problem for the limiting linear system (\ref{Eq:ModelExampleLimit}) was studied in \cite{Kutoyants}. In \cite{Kutoyants}, the author develops MLE estimators for $\theta$ based on (\ref{Eq:ModelExampleLimit}), i.e. using as data $\bar{\mathcal{Y}}_{t}$. However, the difference of our setup with the rest of the literature is that  we want to estimate $\theta$ based on observations $\mathcal{Y}^{\delta}_{t}$, which come from the multiscale model (\ref{Eq:ModelExamplePrelimit}) and not from the limit model (\ref{Eq:ModelExampleLimit}). Of course, the limit problem is used in order to derive properties of the estimators, but the actual inference is done based on observations from the multiscale model.

Let us write $\bar{a}(\theta)=a(\theta)\bar{\lambda}$, $\bar{\beta}(\theta)=\beta(\theta)\bar{q}$. Notice that in the notation of Section \ref{S:ProblemFormulation} we have $\bar{h}_{\theta}(u)=\bar{a}(\theta) u$ and $\bar{g}_{\theta}(u)=-\bar{\beta}(\theta) u$. Let us compute the Fisher information matrix $I(\alpha)$ for this model and derive the conditions under which $I(\alpha)$ is strictly positive and the model is identifiable.
Let us first denote $\hat{U}_{t}=\mathbb{E}[U_{t}|\mathcal{\bar{Y}}_{t}]$. It is known that $\hat{U}_{t}$ satisfies the equation
\begin{align}
d\hat{U}_{t}&=-\bar{\beta}(\theta)\hat{U}_{t}dt+\frac{\bar{a}(\theta)\hat{\sigma}_{t}(\theta)}{\Sigma^{2}}\left(d\bar{Y}_{t}-\bar{a}(\theta)\hat{U}_{t}dt\right)\ ,\label{Eq:Kalman01}
\end{align}
where $\hat{\sigma}_{t}(\theta)=\mathbb{E}(\bar{U}_{t}-\hat{U}_{t})^{2}$  solves the Ricatti equation
\begin{align}
\frac{d\hat{\sigma}_{t}(\theta)}{dt}&=-2\bar{\beta}(\theta)\hat{\sigma}_{t}(\theta)-\frac{\bar{a}(\theta)(\hat{\sigma}_{t}(\theta))^{2}}{\Sigma^{2}}+\gamma^{2}(\theta)\ .\label{Eq:Kalman02}
\end{align}
Next let us define
\[
\zeta(\theta)=\sqrt{\bar{\beta}^{2}(\theta)+\gamma^{2}(\theta)\bar{a}^{2}(\theta)/\Sigma^2}-\bar{\beta}(\theta)=\kappa(\theta)-\bar{\beta}(\theta)\ .
\]

It is easy to see that if $\hat{\sigma}_{0}(\theta)=\frac{\Sigma^{2}}{\bar{a}^{2}(\theta)}\zeta(\theta)$, then $\hat{\sigma}_{t}(\theta)=\frac{\Sigma^{2}}{\bar{a}^{2}(\theta)}\zeta(\theta)$ for all $t\geq 0$, which implies that $\frac{\Sigma^{2}}{\bar{a}^{2}(\theta)}\zeta(\theta)$ is a stationary solution to (\ref{Eq:Kalman02}). If on the other hand
$\hat{\sigma}_{0}(\theta)\neq\frac{\Sigma^{2}}{\bar{a}^{2}(\theta)}\zeta(\theta)$ then $\hat{\sigma}_{t}(\theta)$ converges exponentially fast to $\frac{\Sigma^{2}}{\bar{a}^{2}(\theta)}\zeta(\theta)$, see Section 3.1.1 of \cite{Kutoyants}. Hence, in order to simplify things, let us  assume that (\ref{Eq:Kalman01}) and (\ref{Eq:Kalman02}) are in the stationary regime. In this case, if the initial distribution of $\bar{U}_{0}$ is $N\left(\bar{\pi}^{\theta}_{0}[\bar{h}_{\theta}],\frac{\Sigma^{2}}{\bar{a}^{2}(\theta)}\zeta(\theta)\right)$ then $\bar{U}_{t}\sim N\left(\bar{\pi}^{\theta}_{0}[\bar{h}_{\theta}],\frac{\Sigma^{2}}{\bar{a}^{2}(\theta)}\zeta(\theta)\right)$ for all $t\geq 0$. In this case $\bar{\pi}^{\theta}_{t}[\bar{h}_{\theta}]=\bar{a}(\theta)\hat{\bar U}_{t}$ will satisfy the equation
\begin{eqnarray}
d\bar{\pi}^{\theta}_{t}[\bar{h}_{\theta}]&=&-\bar{\beta}(\theta)\bar{\pi}^{\theta}_{t}[\bar{h}_{\theta}]dt+
\zeta(\theta)\left(d\bar{Y}_{t}-\bar{\pi}^{\theta}_{t}[\bar{h}_{\theta}]dt\right)\ .\label{Eq:AveragedLinearFilter}
\end{eqnarray}

Now notice that if $\theta=\alpha$ (i.e., the true parameter value), then $\nu_{t}$ defined by $ d\bar{\nu}_{t}=\frac{1}{\Sigma}\left(d\bar{Y}_{t}-\bar{\pi}^{\theta}_{t}[\bar{h}_{\theta}]dt\right)$ is  a $\mathbb{P}_{\alpha}$ Brownian motion. In the general case $\bar{\pi}^{\theta}_{t}[\bar{h}_{\theta}]$ satisfies the averaged linear SDE (\ref{Eq:AveragedLinearFilter}), so when $\theta=\alpha$ we have
\begin{equation*}
\bar{\pi}^{\alpha}_{t}[\bar{h}_{\alpha}]=\bar{\pi}^{\alpha}_{0}[\bar{h}_{\alpha}]e^{-\bar{\beta}(\alpha)t}+
\zeta(\alpha) \Sigma\int_{0}^{t}e^{-\bar{\beta}(\alpha)(t-s)}d\bar{\nu}_{s}\ ,
\end{equation*}
from which it is clear that $\bar{\pi}^{\alpha}_{t}[\bar{h}_{\alpha}]$ is Gaussian with invariant law $N\left(0,\frac{\zeta^{2}(\alpha)\Sigma^2}{2\bar{\beta}(\alpha)}\right)$.

Next, considering the derivative of $\bar{\pi}^{\theta}_{t}[\bar{h}_{\theta}]$ with respect to $\theta$, at $\theta=\alpha$ we find that $\dot{\bar{\pi}}^{\alpha}_{t}[\bar{h}_{\alpha}]$ satisfies the SDE
\begin{equation}
d\dot{\bar{\pi}}^{\alpha}_{t}[\bar{h}_{\alpha}]=\left(-\dot{\bar{\beta}}(\alpha)\bar{\pi}^{\alpha}_{t}[\bar{h}_{\alpha}]-\kappa(\alpha)\dot{\bar{\pi}}^{\alpha}_{t}[\bar{h}_{\alpha}]\right)dt
+\dot{\zeta}(\alpha) \Sigma d\bar{\nu}_{t}\ ,\label{Eq:Kalman1}
\end{equation}
and thus we obtain
\begin{equation}
\dot{\bar{\pi}}^{\alpha}_{t}[\bar{h}_{\alpha}]=\dot{\bar{\pi}}^{\alpha}_{0}[\bar{h}_{\alpha}]e^{-\kappa(\alpha)t}
-\dot{\bar{\beta}}(\alpha)\int_{0}^{t}e^{-\kappa(\alpha)(t-s)}\bar{\pi}^{\alpha}_{s}[\bar{h}_{\alpha}]ds
+\dot{\zeta}(\alpha)\int_{0}^{t}e^{-\kappa(\alpha)(t-s)} \Sigma d\bar{\nu}_{s}\ ,\label{Eq:Kalman2}
\end{equation}
from which Fubini's theorem gives
\begin{equation}
\dot{\bar{\pi}}^{\alpha}_{t}[\bar{h}_{\alpha}]=\int_{0}^{t}e^{-\kappa(\alpha)(t-s)}\left(\dot{\kappa}(\alpha)-\dot{\bar{\beta}}(\alpha)e^{\zeta(\alpha)(t-s)}\right) \Sigma d\bar{\nu}_{s}+o(1)\ .\label{Eq:Kalman3}
\end{equation}
Hence, by direct computation using (\ref{Eq:Kalman3}) we can compute the asymptotic variance of the MLE. In particular, we obtain that in $\mathbb{P}_{\alpha}$ probability
\begin{align}
I(\alpha)&=\lim_{T\rightarrow\infty}\frac{1}{T\Sigma^{2}}\int_{0}^{T}\left|\dot{\bar\pi}_{s}^{\alpha}[\bar{h}_{\alpha}]
\right|^{2}ds\nonumber\\
&=\frac{\dot{\bar{\beta}}^{2}(\alpha)}{2\bar{\beta}(\alpha)}+\frac{\dot{\kappa}^{2}(\alpha)}{2\kappa(\alpha)}-2\frac{\dot{\bar{\beta}}(\alpha)\dot{\kappa}(\alpha)}{\bar{\beta}(\alpha)+\kappa(\alpha)}\ .\label{Eq:FisherInformation}
\end{align}

Let us assume now that
\begin{condition}\label{A:AssumptionLinearProblem}
For any compact $\tilde{\Theta}\subset \Theta$ and for any $\epsilon>0$
\begin{itemize}
\item{$\inf_{\theta\in\tilde{\Theta}}\left(|\dot{\bar{\beta}}(\theta)|+|\dot{\kappa}(\theta)|\right)>0$,}
\item{$\inf_{\theta\in\tilde{\Theta}}\inf_{|\alpha-\theta|>\epsilon}
\left(|\bar{\beta}(\alpha)-\bar{\beta}(\theta)|+|\kappa(\alpha)-\kappa(\theta)|\right)>0$.}
\end{itemize}
\end{condition}
It is then proven in a related case in Section 3.1.1 of \cite{Kutoyants}, that in the specific example Condition \ref{A:AssumptionLinearProblem} implies essentially Condition \ref{A:AssumptionIdentifiability}, i.e., we have identifiability of the model, and that  the asymptotic variance is strictly positive, i.e. $I(\alpha)>0$.

Let us next present some simulation studies based on the model problem (\ref{Eq:ModelExamplePrelimit}). We consider three different examples. The examples lack the boundedness of assumption on $h_{\theta}$ indicating the broader applicability of the theoretical results of this paper.
\subsection{Simulation Example 1} Let the processes be scalars $(Y^\delta,U^\delta,X^\delta)\in\mathbb R\times\mathbb R\times\mathbb R$, and consider the following example of the system in \eqref{Eq:ModelExamplePrelimit}:
\begin{eqnarray}
dY^{\delta}_{t}&=&e^{X^{\delta}_{t}} U^{\delta}_{t}dt +  \Sigma dW_{t}\nonumber\hspace{3.6cm}\hbox{(observed)}\\
dU^{\delta}_{t}&=& -U^{\delta}_{t}dt + dV_{t}\nonumber\hspace{4.4cm}\hbox{(hidden)}\\
dX^{\delta}_{t}&=& \frac{1}{\delta}\left(\theta-X^{\delta}_{t}\right)dt + \frac{\sigma}{\sqrt{\delta}} dB_{t}\label{Eq:ModelExamplePrelimitSpecific}\hspace{2.5cm}\hbox{(hidden)}\ .
\end{eqnarray}
A key feature of this example is that the process $X^\delta$ behaves like a multiplicative noise factor. Figure \ref{fig:example1data} shows a realization of the system for a parameterization having $T=25$, $\delta=0.01$, $\Sigma=\sigma=0.1$, along with a true $\alpha=0$ and discrete time step $\Delta t=.02$.

\begin{figure}[htbp] 
   \centering
   \includegraphics[width=4in]{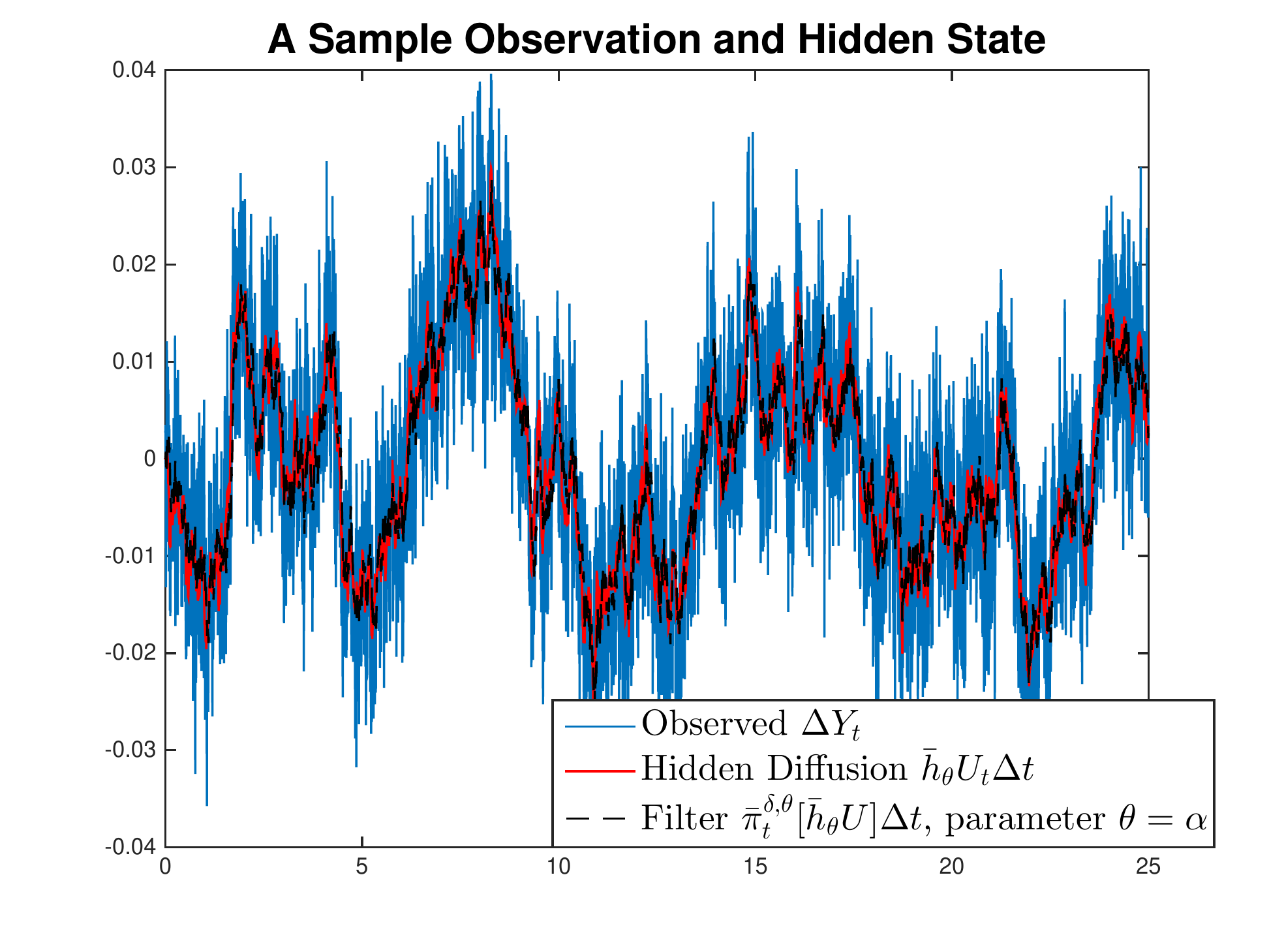}
   \caption{The sample realized from the system given in \eqref{Eq:ModelExamplePrelimitSpecific} with true parameter $\alpha = 0$.}
   \label{fig:example1data}
\end{figure}

In this case the invariant measure $\mu_{\theta}(dx)$ of the $X$ process with $\delta=1$ is that corresponding to $N(\theta,\sigma^{2}/2)$. Notice also that we can compute the Fisher information from (\ref{Eq:FisherInformation}) in closed form and obtain
\begin{align*}
I(\alpha)&=\frac{e^{4\alpha+\sigma^{2}}}{2\Sigma^4(1+e^{2\alpha+\sigma^{2}/2}/\Sigma^2)^{3/2}}.
\end{align*}
Table \ref{T:CLTvar1} presents the standard error for the empirical error for estimator alongside the predicted error from the Fisher information. The table shows a comparison for different values of the true parameter.
\begin{table}
\centering
\begin{tabular}{|c|c|c|c|}
\multicolumn{4}{c}{Statistics for different values of the true parameter for $\theta$.}\\
\hline
$\theta$& estimator &empirical std-err.& theoretical std.err\\
\hline
0&-0.0170&0.0982&0.0900\\
1&0.9844&0.0618&0.0542\\
1.5&1.4591&0.0503& 0.0422\\
\hline
\end{tabular}
\vspace{.2cm}
\caption{Model \eqref{Eq:ModelExamplePrelimitSpecific}, 500 simulations  computed with $T=25$, $\delta=0.01$, $\Sigma=\sigma=0.1$. This table shows the estimator, the empirical standard error and the standard error predicted by Theorem \ref{T:CLTReducedLikelihood}.}
\label{T:CLTvar1}
\end{table}
In Figure \ref{fig:example1hists} we present the histograms for the three different cases of true value of the $\alpha$ parameter, together with the fitted theoretical normal curve as this is given by Theorem \ref{T:CLTReducedLikelihood}.

\begin{figure}[htbp]
\centering
\begin{tabular}{cc}
	\includegraphics[width=3in]{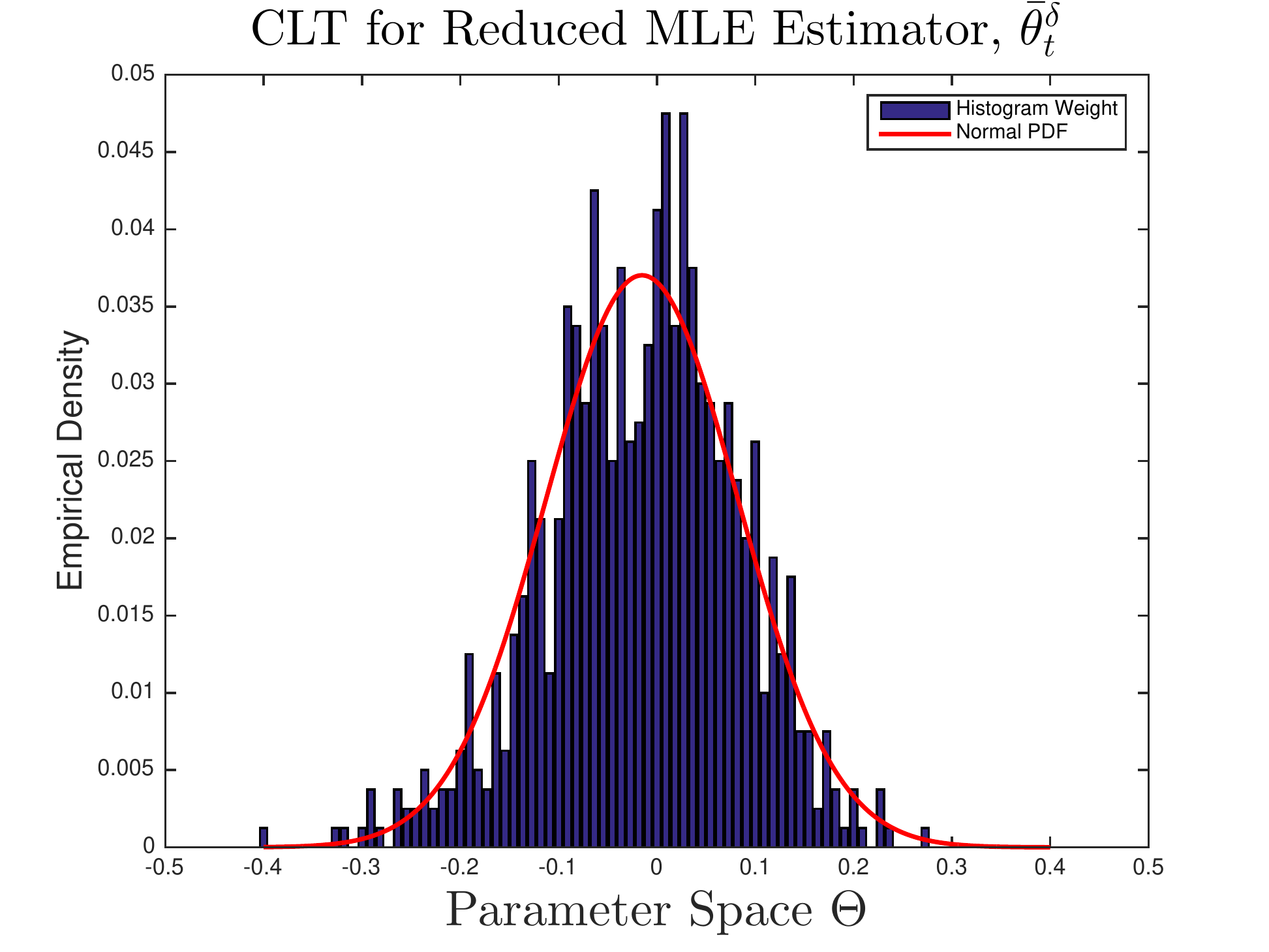}&
	\includegraphics[width=3in]{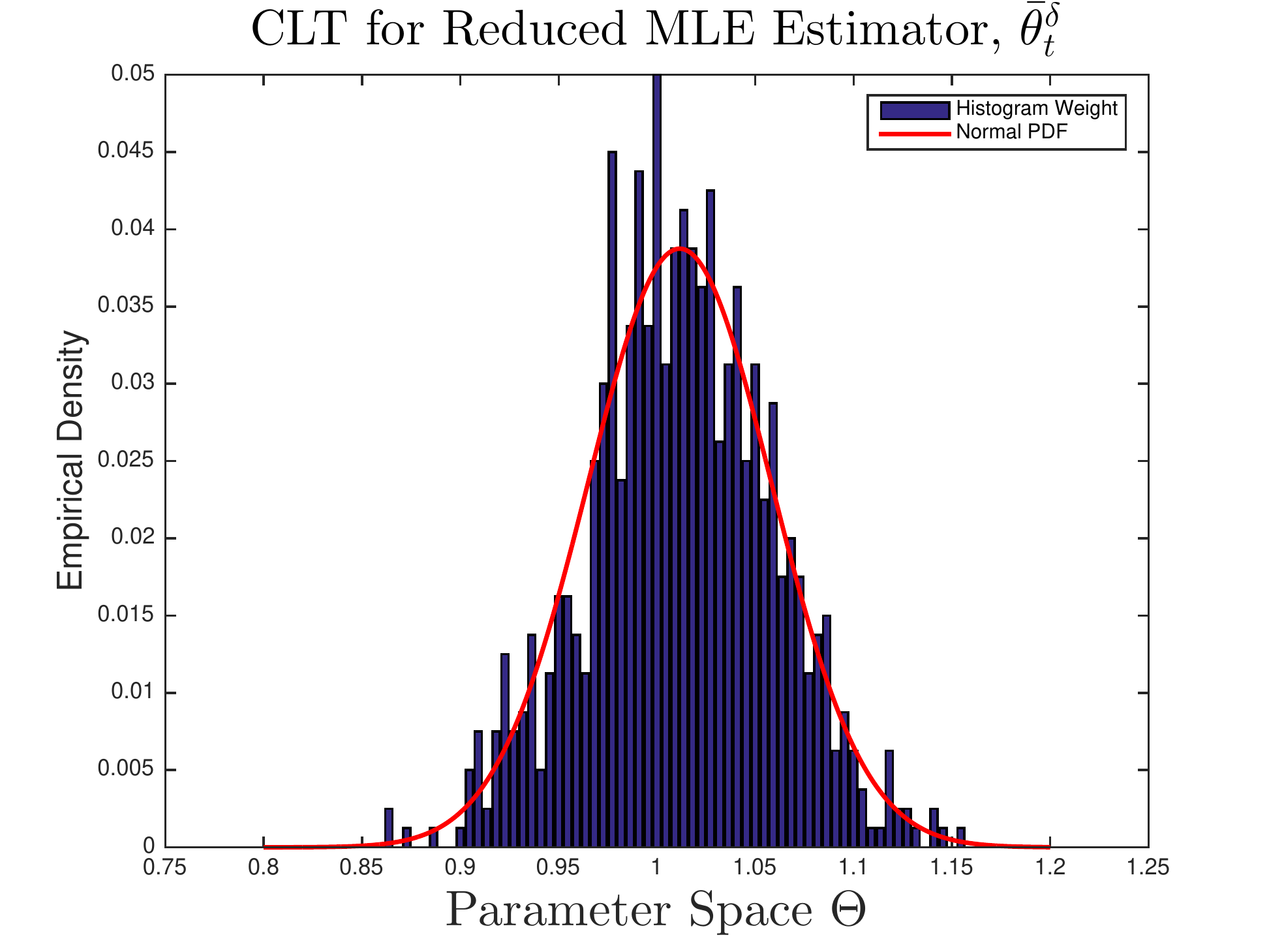}\\
\multicolumn{2}{c}{\includegraphics[width=3in]{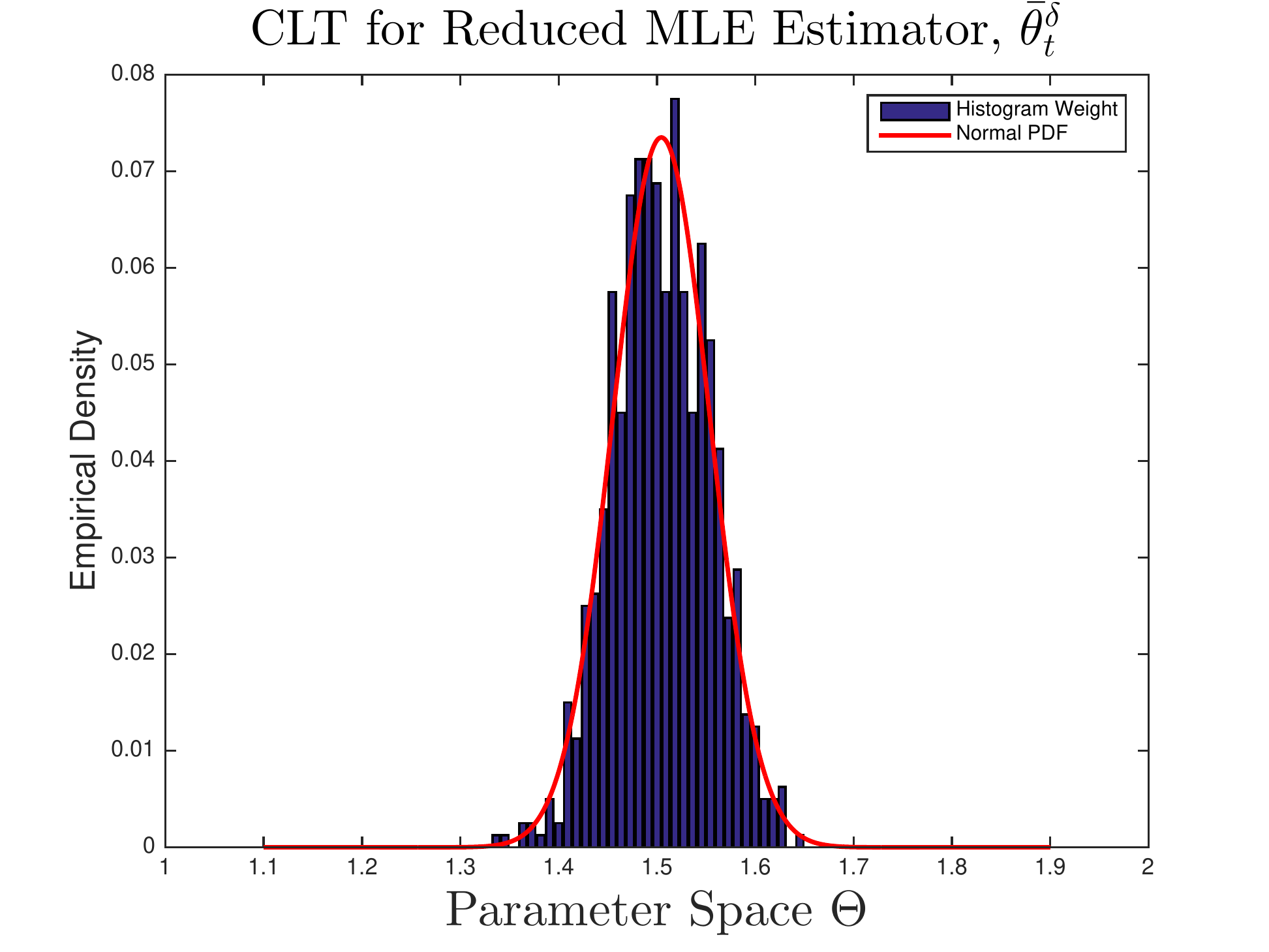}}
\end{tabular}
\caption{\textbf{Top Left:} $\alpha=0$. \textbf{Top Right:} $\alpha=1$. \textbf{Bottom:} $\alpha=1.5$.}
\label{fig:example1hists}
\end{figure}


\subsection{Simulation Example 2} Let the processes be scalars $(Y^\delta,U^\delta,X^\delta)\in\mathbb R\times\mathbb R\times\mathbb R$, and consider the following example of the system in \eqref{Eq:ModelExamplePrelimit}:
\begin{eqnarray}
dY^{\delta}_{t}&=&U^{\delta}_{t}dt +  \Sigma dW_{t}\nonumber\hspace{4.0cm}\hbox{(observed)}\\
dU^{\delta}_{t}&=& -e^{X^{\delta}_{t}} U^{\delta}_{t}dt + dV_{t}\nonumber\hspace{3.8cm}\hbox{(hidden)}\\
dX^{\delta}_{t}&=& \frac{1}{\delta}\left(\theta-X^{\delta}_{t}\right)dt + \frac{\sigma}{\sqrt{\delta}} dB_{t}\label{Eq:ModelExamplePrelimitSpecific2}\hspace{2.5cm}\hbox{(hidden).}
\end{eqnarray}
A key feature of this example is that the process $X^\delta$ affects the mean reversion rate of the $U^\delta$ process.  Figure \ref{fig:example2data} shows a realization of the system for a parameterization having $T=25$, $\delta=0.01$, $\Sigma=\sigma=0.1$, along with a true $\alpha=0$ and discrete time step $\Delta t=.02$.

\begin{figure}[htbp] 
   \centering
   \includegraphics[width=4in]{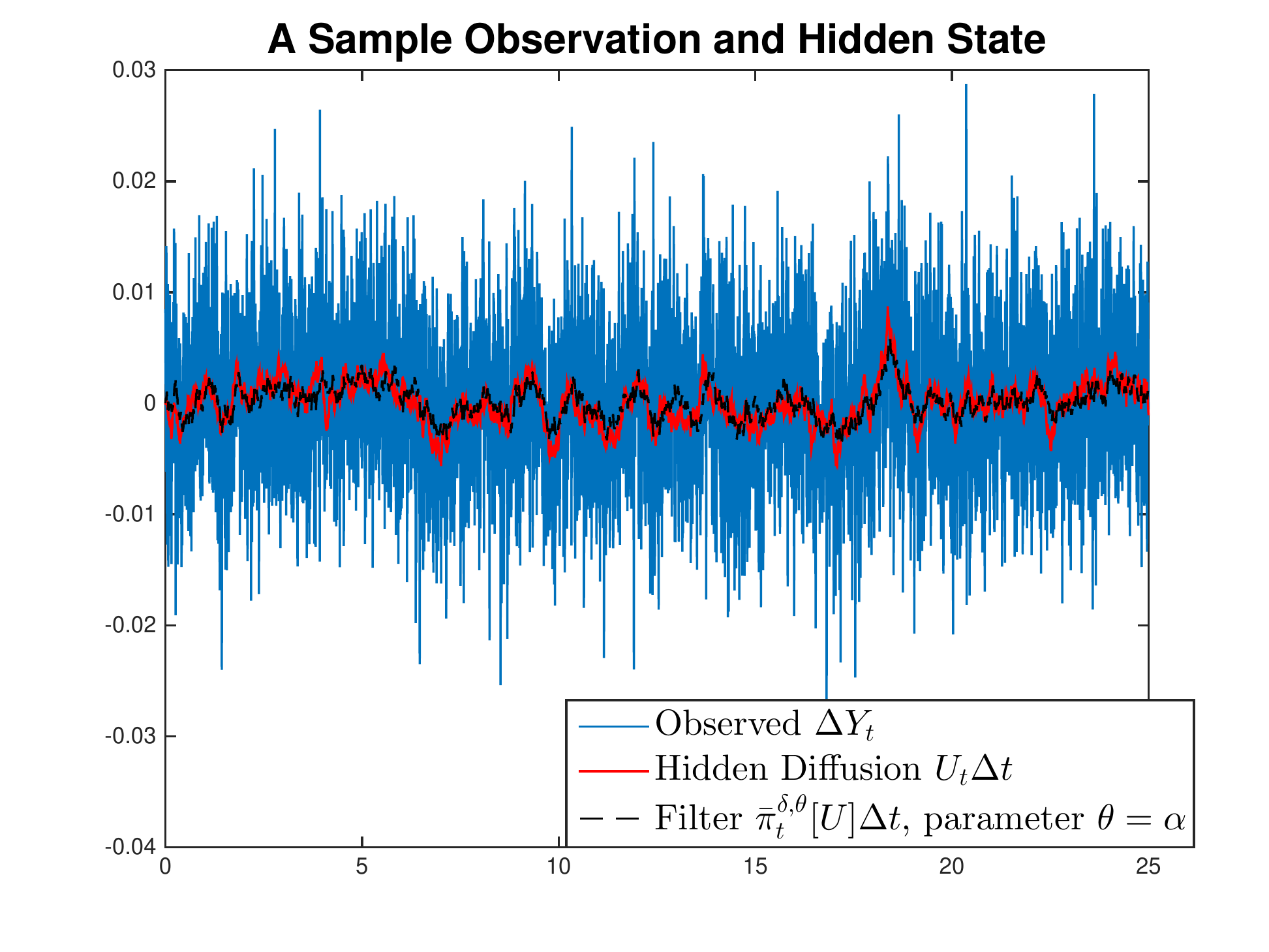}
   \caption{The sample realized from the system given in \eqref{Eq:ModelExamplePrelimitSpecific2} with true parameter $\alpha = 0$.}
   \label{fig:example2data}
\end{figure}

As in the case of Example 1,  the invariant measure $\mu_{\theta}(dx)$ of the $X$ process with $\delta=1$ is that corresponding to $N(\theta,\sigma^{2}/2)$. Notice also that we can compute the Fisher information from \eqref{Eq:FisherInformation} in closed form and obtain
\begin{align*}
I(\alpha)&=\frac{e^{\alpha+\frac{\sigma^2}{4}}}{2}+ \frac{e^{4\alpha+\sigma^2}}{2\left(1/\Sigma^2+e^{2\alpha+\frac{\sigma^2}{2}}\right)^{3/2}}-2\frac{e^{3\alpha+\frac{3\sigma^2}{4}}}{\sqrt{1/\Sigma^2+e^{2\alpha+\frac{\sigma^2}{2}}}\left(e^{\alpha+\frac{\sigma^2}{4}}+\sqrt{1/\Sigma^2+e^{2\alpha+\frac{\sigma^2}{2}}}\right)}
\end{align*}

\begin{table}
\centering
\begin{tabular}{|c|c|c|c|}
\multicolumn{4}{c}{Statistics for different values of the true parameter for $\theta$.}\\
\hline
$\theta$& estimator &empirical std-err.& theoretical std.err\\
\hline
0.5&0.5396&0.2385& 0.2303\\
1&1.0268&0.1943& 0.1917\\
1.5&1.5346&0.1815& 0.1734\\
\hline
\end{tabular}
\vspace{.2cm}
\caption{Model \eqref{Eq:ModelExamplePrelimitSpecific2}, 500 simulations  computed with $T=25$, $\delta=0.01$, $\Sigma=\sigma=0.1$. This table shows the estimator, the empirical standard error and the standard error predicted by Theorem \ref{T:CLTReducedLikelihood}.}
\label{T:CLTvar2}
\end{table}
Table \ref{T:CLTvar2} shows the estimator's standard deviation and the theoretical prediction for various values of the true parameter. In Figure \ref{fig:example2hists} we present the histograms for the three different cases of true value of the $\alpha$ parameter, together with the fitted theoretical normal curve as this is given by Theorem \ref{T:CLTReducedLikelihood}.

\begin{figure}[htbp]
\centering
\begin{tabular}{cc}
	\includegraphics[width=3in]{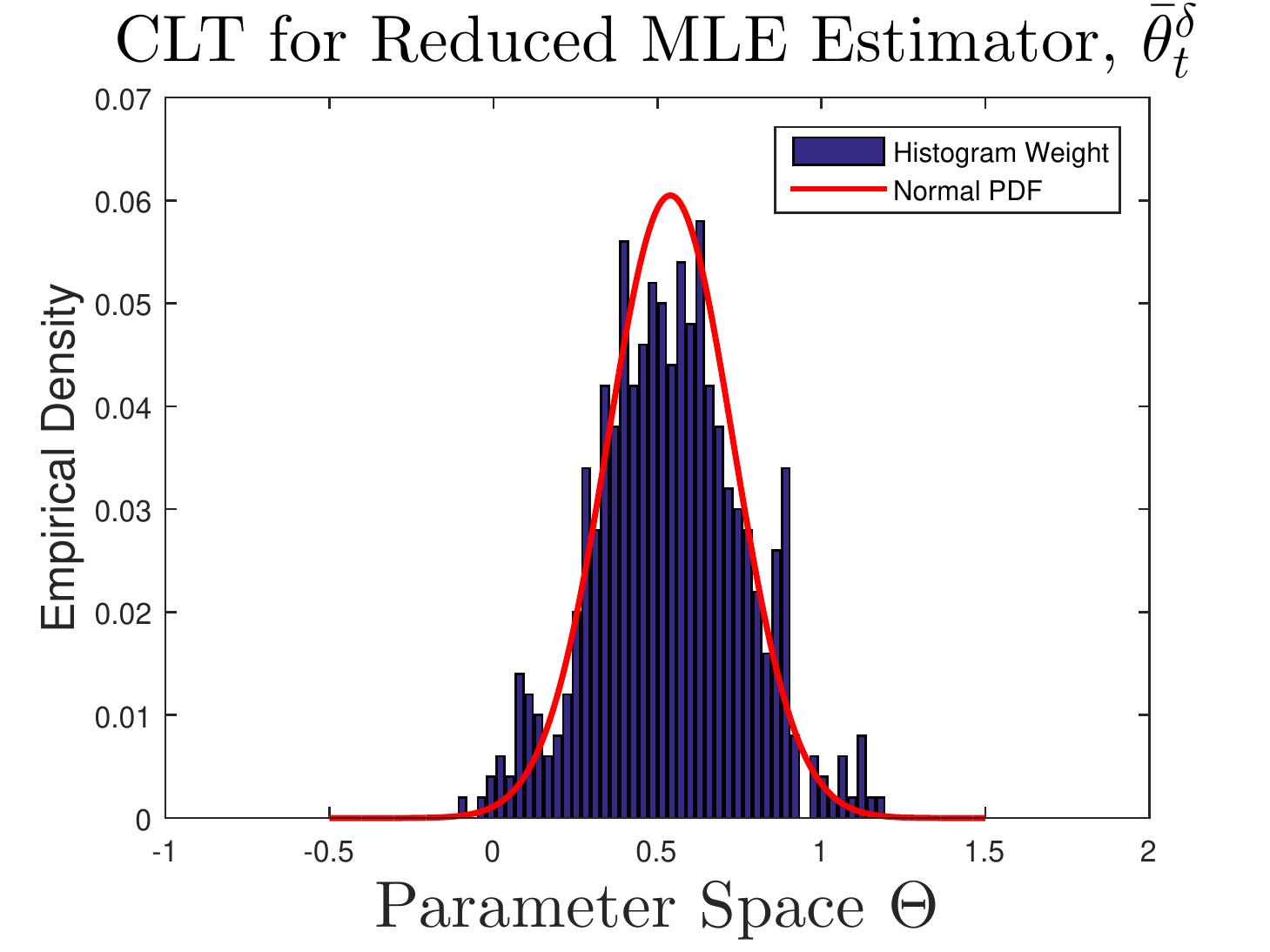}&
	\includegraphics[width=3in]{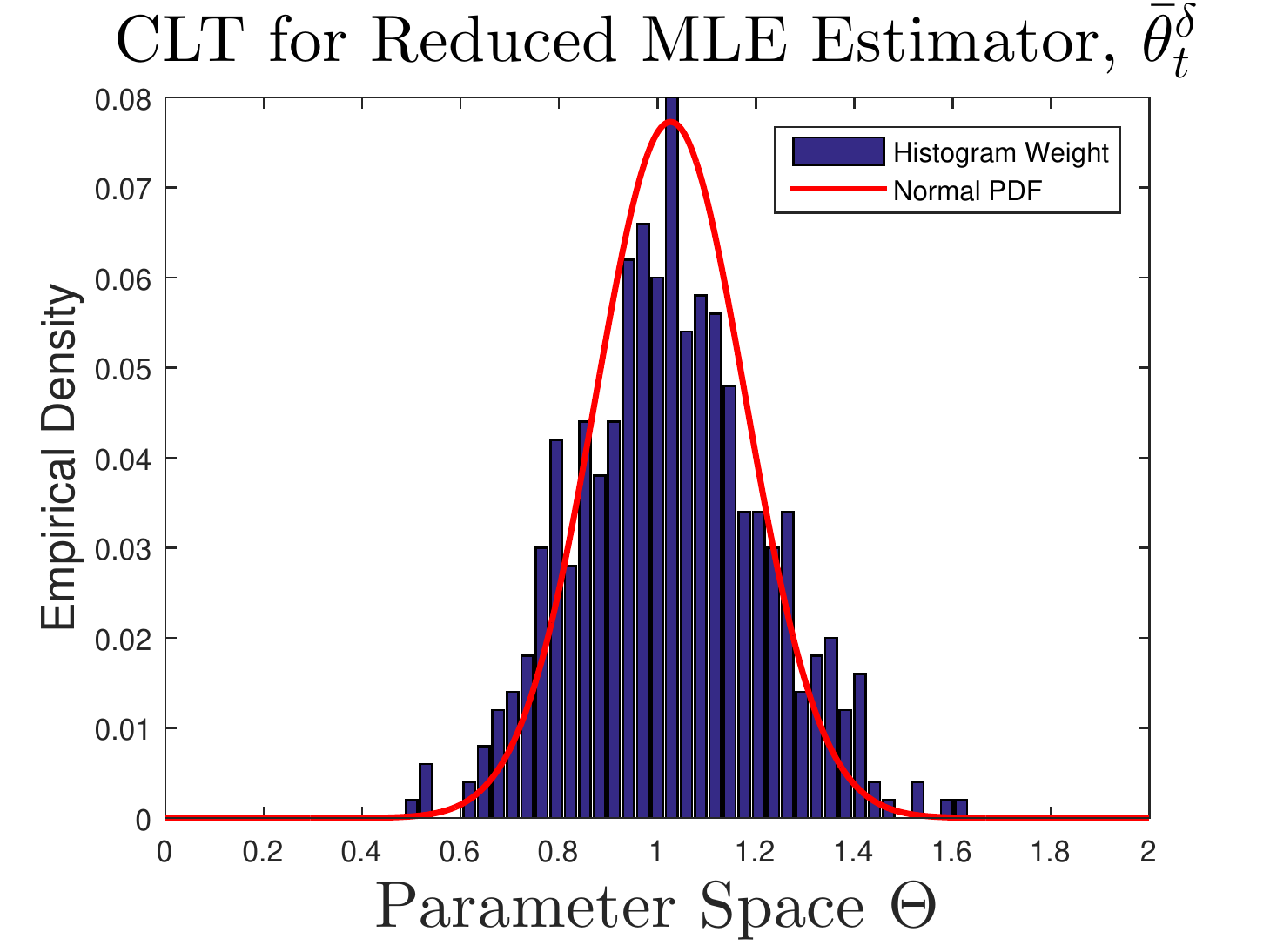}\\
\multicolumn{2}{c}{\includegraphics[width=3in]{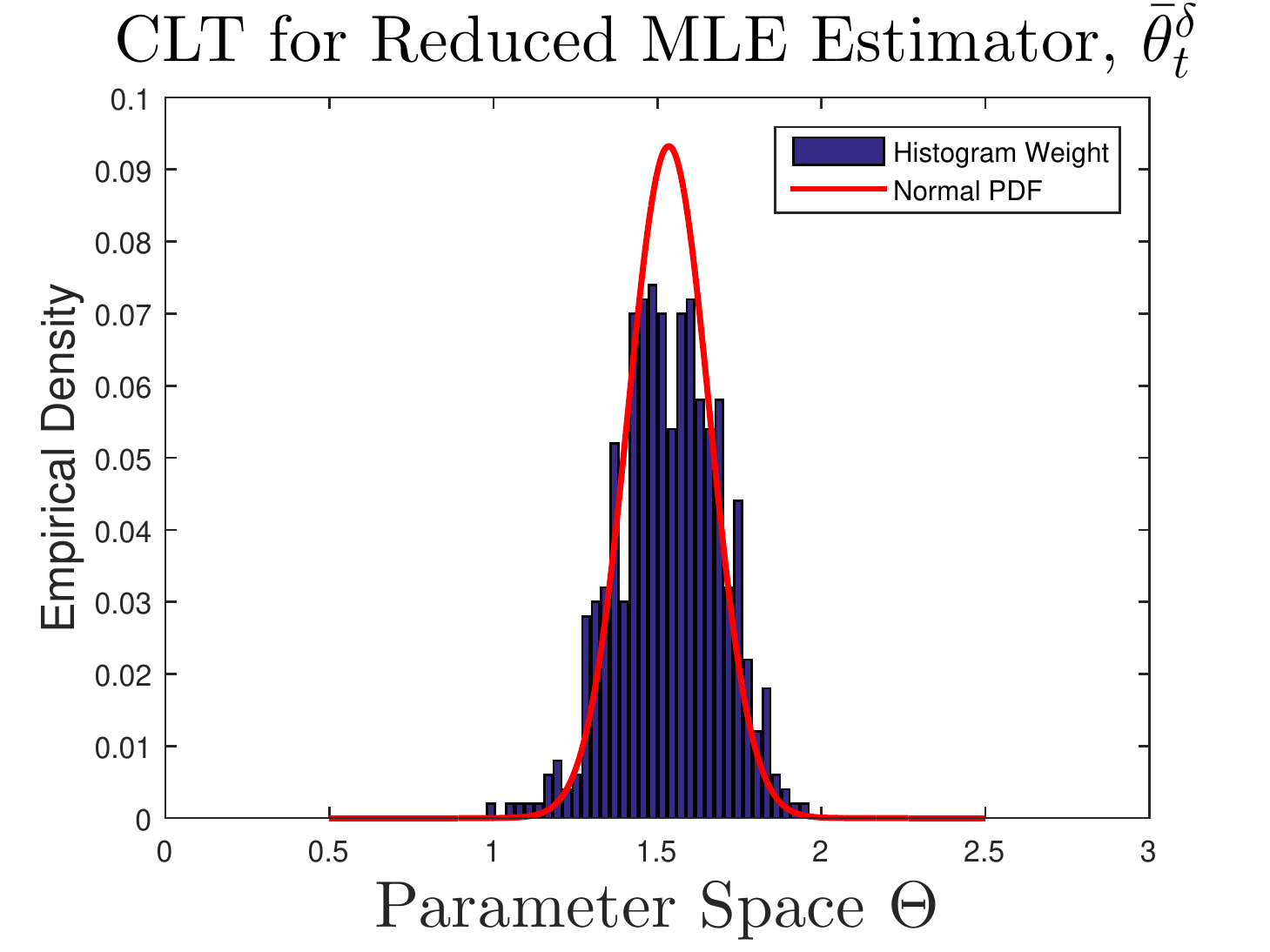}}
\end{tabular}
\caption{\textbf{Top Left:} $\alpha=0$. \textbf{Top Right:} $\alpha=1$. \textbf{Bottom:} $\alpha=1.5$.}
\label{fig:example2hists}
\end{figure}

\subsection{Simulation Example 3} Let the processes be scalars $(Y^\delta,U^\delta,X^\delta)\in\mathbb R\times\mathbb R\times\mathbb R$, and consider the following system:

\begin{eqnarray}
dY^{\delta}_{t}&=& X^{\delta}_{t}dt + \Sigma dW_{t}\nonumber\hspace{4.0cm}\hbox{(observed)}\\
dX^{\delta}_{t}&=& \frac{1}{\delta}\left( U_{t}-X^{\delta}_{t}\right)dt + \frac{\sigma}{\sqrt{\delta}} dB_{t}\label{Eq:ModelExamplePrelimitSpecific3}\hspace{2.5cm}\hbox{(hidden),}
\end{eqnarray}
where $U$ is a continuous time Markov chain taking values in $\{0,1\}$ with transition intensity $\theta>0$, that is
 \begin{align*}
 \frac{d}{dt}
 \begin{pmatrix}\mathbb P(U_t = 0)\\\mathbb P(U_t = 1)\end{pmatrix}
  &= \theta\begin{pmatrix}-1&1\\1&-1\end{pmatrix}
   \begin{pmatrix}\mathbb P(U_t = 0)\\\mathbb P(U_t = 1)\end{pmatrix}\ .
\end{align*}
This example is different from the system in \eqref{Eq:ModelExamplePrelimit} because the $U$ process is not a diffusion; it is comparable to the model considered in \cite{ParkRozovskySowers2010}. Indeed, $U_{t}$ is a discrete space Markov chain and not a continuous diffusion, and so the theory of this paper does not apply, but we conjecture that things can be worked out to find analogous results. Figure \ref{fig:example3data} shows a realization and the filter for this example.
\begin{figure}[htbp] 
   \centering
   \includegraphics[width=4in]{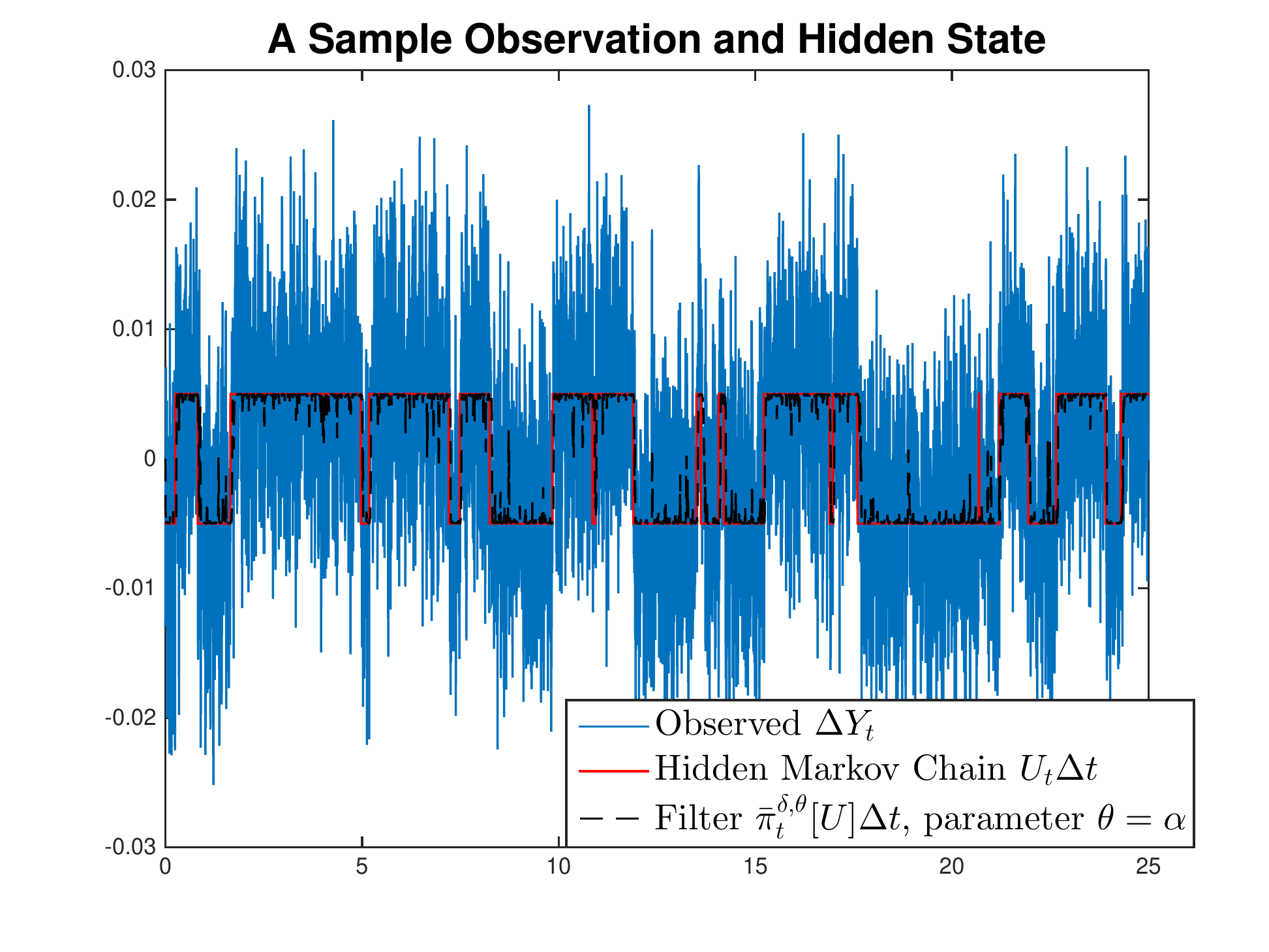}
   \caption{The sample realized from the system given in \eqref{Eq:ModelExamplePrelimitSpecific3} with true parameter $\alpha = 0.7$.}
   \label{fig:example3data}
\end{figure}
Table \ref{T:CLTvar3} shows the estimator's standard deviation and the theoretical prediction for various values of the true parameter, but in this case we have no close-form expression for the Fisher information, so instead we have calculated it numerically.
\begin{table}
\centering
\begin{tabular}{|c|c|c|c|}
\multicolumn{4}{c}{Statistics for different values of the true parameter for $\theta$.}\\
\hline
$\theta$& estimator &empirical std-err.& numerically calculated std.err\\
\hline
0.7&0.7349&0.2305&0.1917\\
1&1.0469&0.2697&0.2253\\
1.8&1.8990&0.3968&0.3058\\
\hline
\end{tabular}
\vspace{.2cm}
\caption{Model \eqref{Eq:ModelExamplePrelimitSpecific3}, 500 simulations  computed with $T=25$, $\delta=0.01$, $\Sigma=\sigma=0.1$. This table shows the estimator and the empirical standard error.}
\label{T:CLTvar3}
\end{table}
Then, in Figures \ref{fig:example3hists} we present the histograms for the three different cases of true value of the $\alpha$ parameter, together with the fitted empirical normal curve. Since, out theory does not cover this case, we cannot provide the theoretical variance of the estimator. However, as the numerical simulations indicate, a central limit theorem is expected to hold.

\begin{figure}[htbp]
\centering
\begin{tabular}{cc}
	\includegraphics[width=3in]{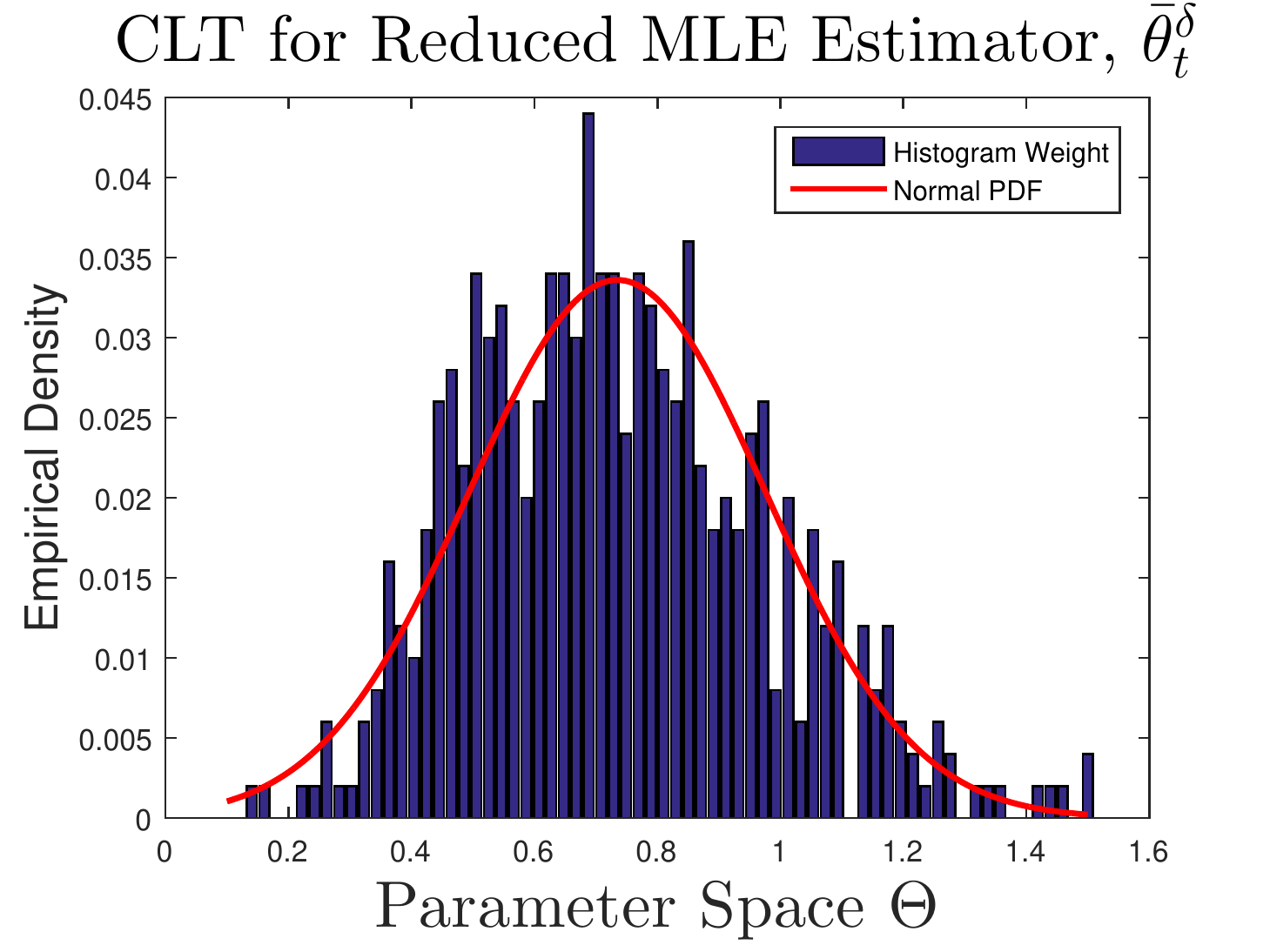}&
	\includegraphics[width=3in]{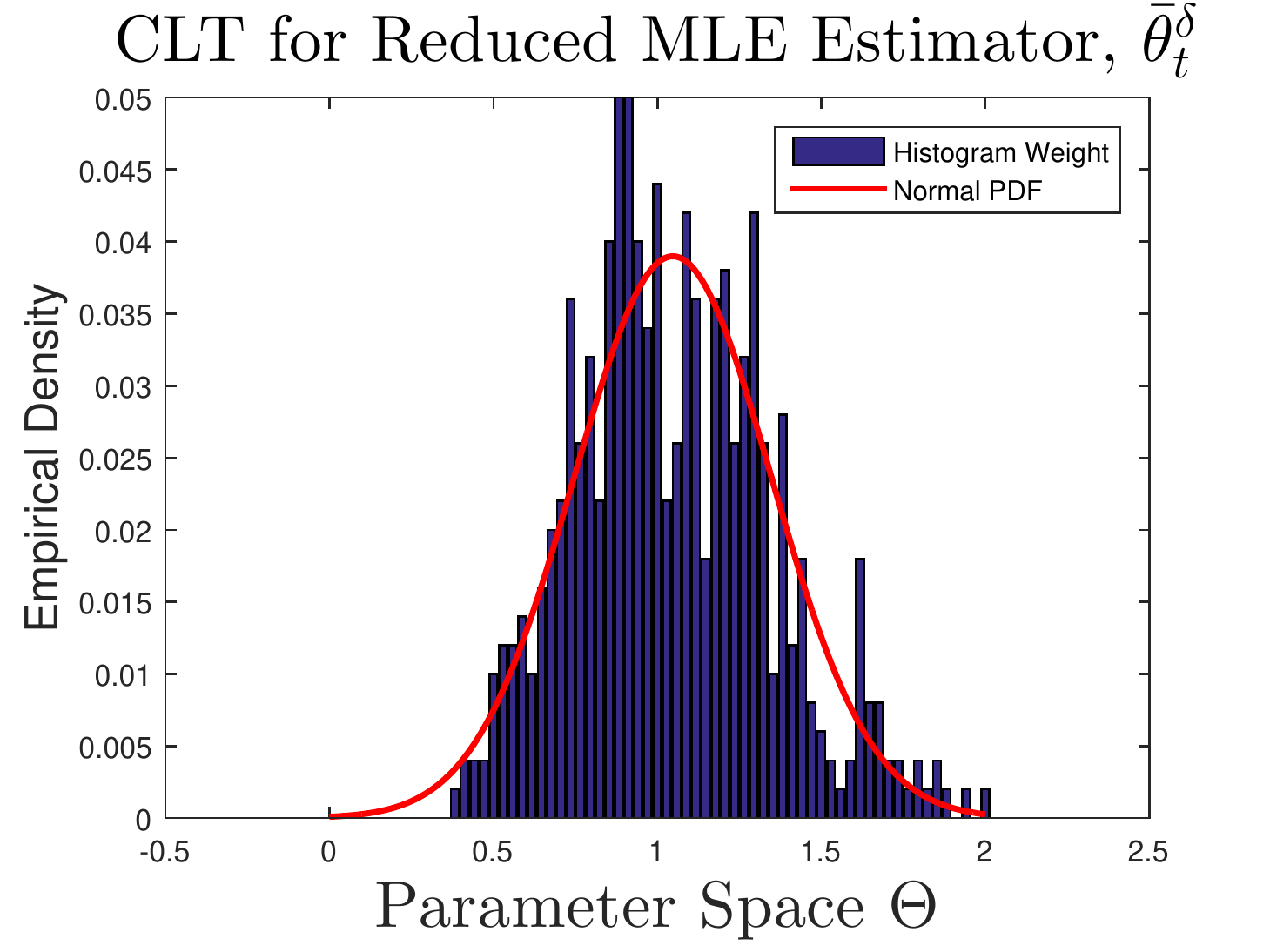}\\
\multicolumn{2}{c}{\includegraphics[width=3in]{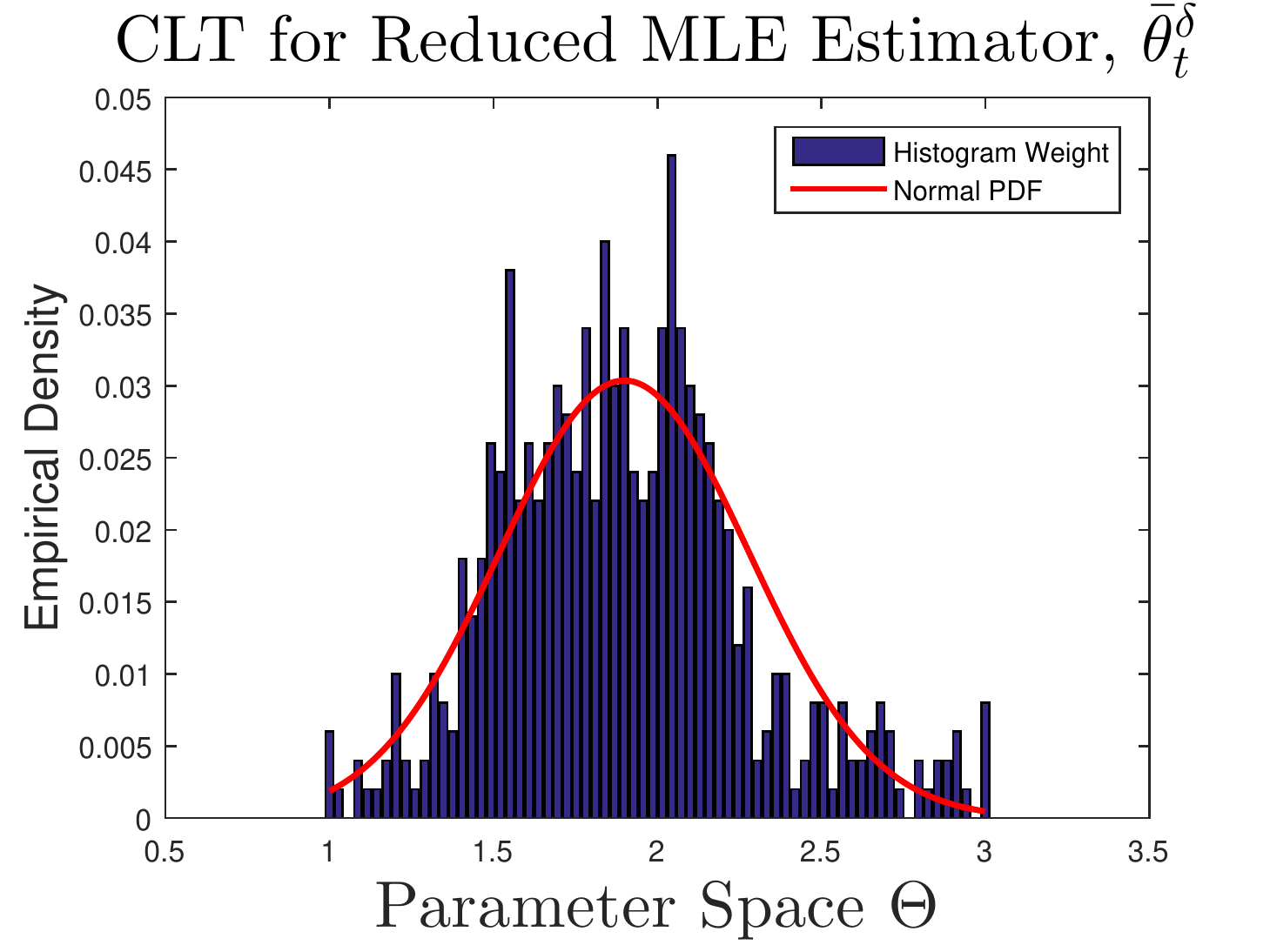}}
\end{tabular}
\caption{\textbf{Top Left:} $\alpha=.7$. \textbf{Top Right:} $\alpha=1$. \textbf{Bottom:} $\alpha=1.8$.}
\label{fig:example3hists}
\end{figure}


\appendix

\section{Proof of Theorem \ref{T:FilterConvergence1}}
\label{A:FilterConvergence1}

Versions of Theorem \ref{T:FilterConvergence1} in simpler settings have appeared in the literature in  \cite{ImkellerSriPerkowskiYeong2012, ParkSriSowers2008, ParkRozovskySowers2010,ParkSriSowers2011}. The first main difference that Theorem \ref{T:FilterConvergence1} has when compared to the previous works is that
under the measure parameterized  by the true parameter value (i.e. the measure under which the observations are made) the filters will converge for \textit{any} parameter value.  Moreover, the second main difference  is that we need to prove that the convergence of the filters is for test functions in the space the space $\mathcal{A}_{\eta}^{\theta}$, whereas the results in \cite{ImkellerSriPerkowskiYeong2012} use bounded and smooth test functions that depend only on the slow motion $U$.

\begin{lemma}\label{L:NeededErgodicResult}
Assume Condition \ref{A:Assumption1}. Let us consider a copy of $(\widetilde{X}^{\delta},\widetilde U^\delta,\bar{\widetilde{U}})$, which has the same law as $(X^{\delta},U^\delta,\bar U)$, but which is independent of $(X^{\delta},U^\delta,\bar U)$. Then, we have
\begin{align}
\lim_{\delta\downarrow 0}\mathbb E_\theta^*\left|\exp\left(\int_0^t h_{\theta}(\tilde X_s^\delta,\tilde U_s^\delta)h_{\theta}(X_s^\delta, U_s^\delta)ds\right)-\exp\left(\int_0^th_{\theta}(\tilde{X}_s^\delta, \tilde{U}_s^\delta) \bar h_\theta(\bar  U_s)ds\right)\right|&=0\nonumber
\end{align}
and
\begin{align}
\lim_{\delta\downarrow 0}\mathbb E_\theta^*\left|\exp\left(\int_0^t\bar h_\theta(\bar U_s)\bar h_\theta(\bar{\widetilde{U}}_s)ds\right)- \exp\left(\int_0^t h(X_s^\delta, U_s^\delta) \bar h_\theta(\bar{\widetilde{U}}_s)ds\right)\right|&=0.\nonumber
\end{align}
\end{lemma}

\begin{proof}
We will only prove the first statement as the proof of the second statement is the same. The convexity inequality $|e^a-e^b|\leq (e^a+e^b)|a-b|$ is used to obtain
\begin{align}
&\mathbb E_\theta^*\left|\exp\left(\int_0^t h_{\theta}(\tilde X_s^\delta,\tilde U_s^\delta)h_{\theta}(X_s^\delta, U_s^\delta)ds\right)-\exp\left(\int_0^th_{\theta}(\tilde{X}_s^\delta, \tilde{U}_s^\delta) \bar h_\theta(\bar  U_s)ds\right)\right|\nonumber\\
&\leq\mathbb E_\theta^*\left(\exp\left(\int_0^t h_{\theta}(\tilde X_s^\delta,\tilde U_s^\delta)h_{\theta}(X_s^\delta, U_s^\delta)ds\right)+\exp\left(\int_0^th_{\theta}(\tilde{X}_s^\delta, \tilde{U}_s^\delta) \bar h_\theta(\bar  U_s)ds\right)\right)|\Xi_{t}^{\delta}|\nonumber\\
&\leq 2e^{T\|h_\theta\|_\infty^2}\mathbb E_\theta^*|\Xi_{t}^{\delta}|\ ,\nonumber
\end{align}
where we have defined the process
\begin{align*}
\Xi_{t}^{\delta}&=\int_0^th_{\theta}(\tilde X_s^\delta, \tilde U_s^\delta) \Big(h_{\theta}( X_s^\delta, U_s^\delta)-\bar h_\theta(\bar  U_s)\Big)ds\ .
\end{align*}
It remains to show that this term goes to zero. By the dominated convergence theorem and the independence of pairs $(X,U)$ and $(\tilde{X},\tilde{U})$, the ergodic theory applies to the joint process $(X,U,\tilde{X},\tilde{U})$, and in particular we see that the limit of $\Xi$ is
\begin{align}
\mathbb E_\theta^* |\Xi_{t}^{\delta}|&= \mathbb E_\theta^* \left|\int_0^t h_{\theta}(\tilde X_s^\delta, \tilde U_s^\delta) \Big(h_{\theta}( X_s^\delta, U_s^\delta)-\bar h_\theta(\bar  U_s)\Big)ds\right|\rightarrow 0, \text{ as }\delta\downarrow 0\ .\nonumber
\end{align}
This proves convergence in $L^{1}$. Convergence in $L^{2}$ follows again from dominated convergence theorem, since $|\Xi_{t}^{\delta}|^{2}\leq 2T^{2}\|h_{\theta}\|^{2}_{\infty}$, concluding the proof of the lemma.
\end{proof}

\begin{lemma}
 \label{L:FilterConvergence3}
Let us consider bounded $f:\mathcal X\times\mathcal U\rightarrow\mathbb R$ and assume Condition \ref{A:Assumption1}. For any $\theta\in\Theta$, we have uniformly in $t\in[0,T]$
\[\mathbb E_\theta^*\left|\phi_t^{\delta,\theta}[f]-\phi_t^{\delta,\theta}[\bar f_\theta]\right|^{2}\rightarrow 0\qquad\hbox{as }\delta\rightarrow0\ ,\]
\end{lemma}
where $\bar f_\theta:\mathcal U\rightarrow\mathbb R$ defined as $\bar f_\theta(u) = \int  f(x,u)\mu_\theta(x,u)dx$.
\begin{proof}
Let us consider an independent copy of $(\widetilde{X}^{\delta},\widetilde U^\delta)$, which has the same law as $(X^{\delta},U^\delta)$, but which is independent of $(X^{\delta},U^\delta,W)$. We have
\begin{align}
&\mathbb E_\theta^*\left(\phi_t^{\delta,\theta}[f]-\phi_t^{\delta,\theta}[\bar f_\theta]\right)^2\nonumber\\
&=\mathbb E_\theta^*\left(\phi_t^{\delta,\theta}[f-\bar f_\theta]\right)^2\nonumber\\
&= \mathbb E_\theta^*\left[\mathbb E_\theta^*\left[\left(f(X_t^\delta,U_t^\delta)-\bar f_\theta(U_t^\delta)\right)\exp\left( \int_{0}^th_{\theta}(X^{\delta}_{s},U_s^\delta)dY_s^\delta-\frac{1}{2}\int_{0}^t\left|h_{\theta}(X^{\delta}_{s},U_s^\delta)\right|^{2}ds   \right)\Big|\mathcal Y_t^\delta\right]^2\right]\nonumber\\
&= \mathbb E_\theta^*\Bigg[\mathbb E_\theta^*\Bigg[\Big(f(X_t^\delta,U_t^\delta)-\bar f_\theta(U_t^\delta)\Big)\Big(f(\widetilde X_t^\delta,\widetilde U_t^\delta)-\bar f_\theta(\widetilde U_t^\delta)\Big)\nonumber\\
&\times\exp\left( \int_{0}^t\left(h_{\theta}(X^{\delta}_{s},U_s^\delta)+h_{\theta}(\widetilde X^{\delta}_{s},\widetilde U_s^\delta)\right)dY_s^\delta-\frac{1}{2}\int_{0}^t\left(\left|h_{\theta}(X^{\delta}_{s},U_s^\delta)\right|^{2}+\left|h_{\theta}(\widetilde X^{\delta}_{s},\widetilde U_s^\delta)\right|^{2}\right)ds   \right)\Big|\mathcal Y_t^\delta\Bigg]\Bigg]\nonumber\\
&= \mathbb E_\theta^*\Bigg[\mathbb E_\theta^*\Bigg[\Big(f(X_t^\delta,U_t^\delta)-\bar f_\theta(U_t^\delta)\Big)\Big(f(\widetilde X_t^\delta,\widetilde U_t^\delta)-\bar f_\theta(\widetilde U_t^\delta)\Big)\nonumber\\
&\hspace{2cm}\times\exp\left( \int_{0}^t\left(h_{\theta}(X^{\delta}_{s},U_s^\delta)+h_{\theta}(\widetilde X^{\delta}_{s},\widetilde U_s^\delta)\right)dY_s^\delta\right.\nonumber\\
&\hspace{5cm}\left.-\frac{1}{2}\int_{0}^t\left(\left|h_{\theta}(X^{\delta}_{s},U_s^\delta)\right|^{2}+\left|h_{\theta}(\widetilde X^{\delta}_{s},\widetilde U_s^\delta)\right|^{2}\right)ds   \right)\Big|\mathcal F_t^{U,\tilde U,X,\tilde X}\Bigg]\Bigg]\nonumber\\
&= \mathbb E_\theta^*\Bigg[\Big(f(X_t^\delta,U_t^\delta)-\bar f_\theta(U_t^\delta)\Big)\Big(f(\widetilde X_t^\delta,\widetilde U_t^\delta)-\bar f_\theta(\widetilde U_t^\delta)\Big)\exp\left(\int_{0}^{t}h_{\theta}(X^{\delta}_{s},U_s^\delta)h_{\theta}(\widetilde X^{\delta}_{s},\widetilde U_s^\delta)ds   \right)\Bigg] \nonumber\\
&= \mathbb E_\theta^* \Bigg[\Big(f(X_t^\delta,U_t^\delta)-\bar f_\theta(U_t^\delta)\Big)\Big(f(\widetilde X_t^\delta,\widetilde U_t^\delta)-\bar f_\theta(\widetilde U_t^\delta)\Big)\nonumber\\
&\hspace{1cm}\times\Bigg(\exp\left(\int_{0}^{t}h_{\theta}(X^{\delta}_{s},U_s^\delta)h_{\theta}(\widetilde X^{\delta}_{s},\widetilde U_s^\delta)ds\right)-\exp\left(\int_0^th_{\theta}(\tilde{X}_s^\delta, \tilde{U}_s^\delta) \bar h_\theta(\bar  U_s)ds\right)\Bigg)\Bigg]\nonumber\\
&~~+\mathbb E_\theta^*\Bigg[\Big(f(X_t^\delta,U_t^\delta)-\bar f_\theta(U_t^\delta)\Big)\Big(f(\widetilde X_t^\delta,\widetilde U_t^\delta)-\bar f_\theta(\widetilde U_t^\delta)\Big)\exp\left(\int_0^th_{\theta}(\tilde{X}_s^\delta, \tilde{U}_s^\delta) \bar h_\theta(\bar  U_s)ds\right)\Bigg] \label{Eq:UnnormalizedFilter0}\ .
\end{align}
In the 2nd to last line of the above display, the term goes to zero by Lemma \ref{L:NeededErgodicResult},
\begin{align*}
&\Bigg|\mathbb E_\theta^* \Bigg[\Big(f(X_t^\delta,U_t^\delta)-\bar f_\theta(U_t^\delta)\Big)\Big(f(\widetilde X_t^\delta,\widetilde U_t^\delta)-\bar f_\theta(\widetilde U_t^\delta)\Big)\\
&\hspace{1cm}\times\Bigg(\exp\left(\int_{0}^{t}h_{\theta}(X^{\delta}_{s},U_s^\delta)h_{\theta}(\widetilde X^{\delta}_{s},\widetilde U_s^\delta)ds\right)-\exp\left(\int_0^th_{\theta}(\tilde{X}_s^\delta, \tilde{U}_s^\delta) \bar h_\theta(\bar  U_s)ds\right)\Bigg)\Bigg]\Bigg|\\
&\leq 4\|f\|_\infty^2 \mathbb E_\theta^* \Bigg|\exp\left(\int_{0}^{t}h_{\theta}(X^{\delta}_{s},U_s^\delta)h_{\theta}(\widetilde X^{\delta}_{s},\widetilde U_s^\delta)ds\right)-\exp\left(\int_0^th_{\theta}(\tilde{X}_s^\delta, \tilde{U}_s^\delta) \bar h_\theta(\bar  U_s)ds\right) \Bigg|\\
&\rightarrow 0\ ,
\end{align*}
and the term in the last line of the display (\ref{Eq:UnnormalizedFilter0}) goes to zero as follows,

\begin{align*}
&\Bigg|\mathbb E_\theta^*\Bigg[\Big(f(X_t^\delta,U_t^\delta)-\bar f_\theta(U_t^\delta)\Big)\Big(f(\widetilde X_t^\delta,\widetilde U_t^\delta)-\bar f_\theta(\widetilde U_t^\delta)\Big)\exp\left(\int_0^th_{\theta}(\tilde{X}_s^\delta, \tilde{U}_s^\delta) \bar h_\theta(\bar  U_s)ds\right)\Bigg]\Bigg|\\
&\leq 4\|f\|_\infty^2 \mathbb E_\theta^*\Bigg|\exp\left(\int_0^th_{\theta}(\tilde{X}_s^\delta, \tilde{U}_s^\delta) \bar h_\theta(\bar  U_s)ds\right)-\exp\left(\int_0^{t-\epsilon}h_{\theta}(\tilde{X}_s^\delta, \tilde{U}_s^\delta) \bar h_\theta(\bar  U_s)ds\right)\Bigg|\\
&+\Bigg|\mathbb E_\theta^*\Bigg[\mathbb E_\theta^*\Big[f(X_t^\delta,U_t^\delta)-\bar f_\theta(U_t^\delta)\Big|\mathcal F_{t-\epsilon}^{U^{\delta},\bar U}\vee\mathcal F_t^{\tilde U^{\delta},\tilde X^{\delta}}\Big]\Big(f(\widetilde X_t^\delta,\widetilde U_t^\delta)-\bar f_\theta(\widetilde U_t^\delta)\Big)\exp\left(\int_0^{t-\epsilon}h_{\theta}(\tilde{X}_s^\delta, \tilde{U}_s^\delta) \bar h_\theta(\bar  U_s)ds\right)\Bigg]\Bigg|\\
&\leq 4\|f\|_\infty^2 \mathbb E_\theta^*\Bigg|\exp\left(\int_0^th_{\theta}(\tilde{X}_s^\delta, \tilde{U}_s^\delta) \bar h_\theta(\bar  U_s)ds\right)-\exp\left(\int_0^{t-\epsilon}h_{\theta}(\tilde{X}_s^\delta, \tilde{U}_s^\delta) \bar h_\theta(\bar  U_s)ds\right)\Bigg|\\
&\hspace{2cm}+2\|f\|_\infty \mathbb E_\theta^*\Bigg[\Big|\mathbb E_\theta^*\Big[f(X_t^\delta,U_t^\delta)-\bar f_\theta(U_t^\delta)\Big|\mathcal F_{t-\epsilon}^{U^{\delta},\bar U}\vee\mathcal F_t^{\tilde U^{\delta},\tilde X^{\delta}}\Big]\Big|\exp\left(\int_0^{t-\epsilon}h_{\theta}(\tilde{X}_s^\delta, \tilde{U}_s^\delta) \bar h_\theta(\bar  U_s)ds\right)\Bigg]\\
&= 4\|f\|_\infty^2 \mathbb E_\theta^*\Bigg|\exp\left(\int_0^th_{\theta}(\tilde{X}_s^\delta, \tilde{U}_s^\delta) \bar h_\theta(\bar  U_s)ds\right)-\exp\left(\int_0^{t-\epsilon}h_{\theta}(\tilde{X}_s^\delta, \tilde{U}_s^\delta) \bar h_\theta(\bar  U_s)ds\right)\Bigg|\\
&\hspace{2cm}+2\|f\|_\infty \mathbb E_\theta^*\Bigg[\Big|\mathbb E_\theta^*\Big[f(X_t^\delta,U_t^\delta)-\bar f_\theta(U_t^\delta)\Big|\mathcal F_{t-\epsilon}^{U^{\delta},\bar U}\Big]\Big|\exp\left(\int_0^{t-\epsilon}h_{\theta}(\tilde{X}_s^\delta, \tilde{U}_s^\delta) \bar h_\theta(\bar  U_s)ds\right)\Bigg]\\
&\rightarrow 4\|f\|_\infty^2 \mathbb E_\theta^*\Bigg|\exp\left(\int_0^t\bar h_{\theta}(\tilde{\bar U}_s) \bar h_\theta(\bar  U_s)ds\right)-\exp\left(\int_0^{t-\epsilon}\bar h_{\theta}(\tilde{\bar U}_s) \bar h_\theta(\bar  U_s)ds\right)\Bigg|\, \text{ as }\delta\downarrow 0
\end{align*}
where $\epsilon>0$ can be arbitrarily small. The limit is taken as $\delta \to 0$, with the conditional expectation being handled in the following way:
\begin{align*}
&\mathbb E_\theta^*\Big[f(X_t^\delta,U_t^\delta)-\bar f_\theta(U_t^\delta)\Big|\mathcal F_{t-\epsilon}^{U^{\delta},\bar U}\vee\mathcal F_t^{\tilde U^{\delta},\tilde X^{\delta}}\Big]\\
&=\mathbb E_\theta^*\Big[f(X_t^\delta,U_t^\delta)-\bar f_\theta(U_t^\delta)\Big|\mathcal F_{t-\epsilon}^{U^{\delta},\bar U}\Big]\qquad\qquad\hbox{by independence of $(U^{\delta},X^{\delta})$ from $(\tilde U^{\delta},\tilde X^{\delta})$,}\\
&=\mathbb E_\theta^*\Big[\mathbb E_\theta^*\Big[f(X_t^\delta,U_t^\delta)-\bar f_\theta(U_t^\delta)\Big|X^\delta_{t-\epsilon},U^{\delta}_{t-\epsilon}\Big]\Big|\mathcal F_{t-\epsilon}^{U^{\delta},\bar U}\Big]\\
&\rightarrow 0\ ,
\end{align*}
 The last convergence is due to the fact that $\mathbb E_\theta^*\Big[f(X_t^\delta,U_t^\delta)-\bar f_\theta(U_t^\delta)\Big|X_{t-\epsilon}^\delta,U_{t-\epsilon}^\delta\Big]\rightarrow 0$ as $\delta \to 0$ for any $\epsilon>0$ because of the ergodicity implied by Condition \ref{A:Assumption1}. 
Moreover, owing to Condition \ref{A:Assumption1}'s assumptions on $h$ 
\[
\mathbb E_\theta^*\Bigg|\exp\left(\int_0^t\bar h_{\theta}(\tilde{\bar U}_s) \bar h_\theta(\bar  U_s)ds\right)-\exp\left(\int_0^{t-\epsilon}\bar h_{\theta}(\tilde{\bar U}_s) \bar h_\theta(\bar  U_s)ds\right)\Bigg|\rightarrow 0\ ,
\]
as $\epsilon \to 0$ by dominated convergence theorem. Hence, the remaining term is arbitrarily small, and we conclude that all terms converge to zero with $\delta$.
\end{proof}

\begin{lemma}[Lemma 6.7 in \cite{ImkellerSriPerkowskiYeong2012}]\label{L:BoundedZ}
Given that $h_{\theta}$ is bounded, then for any $q\in[1,\infty)$ we have that
\begin{equation*}
 \sup_{t\in[0,T]}\sup_{\delta\in(0,1)}\mathbb{E}_{\theta}^{*}|Z^{\delta,\theta}_{t}|^{q}+
 \sup_{t\in[0,T]}\sup_{\delta\in(0,1)}\mathbb{E}_{\theta}|Z^{\delta,\theta}_{t}|^{-q}<\infty\ .
\end{equation*}
\end{lemma}

\begin{lemma}
 \label{L:FilterConvergence5}
Let us consider bounded $f:\mathcal U\rightarrow\mathbb R$ and assume Condition \ref{A:Assumption1}. For any $\alpha,\theta\in\Theta$, we have uniformly in $t\in[0,T]$
\[\mathbb E_\alpha\left|\pi_t^{\delta,\theta}[f]-\bar\pi_t^{\delta,\theta}[ f]\right|\rightarrow 0\qquad\hbox{as }\delta\rightarrow0\ ,\]
where $\bar \pi_t^{\delta,\theta}[f]=\bar \phi_t^{\delta,\theta}[f]/\bar \pi_t^{\delta,\theta}[1]$ is the averaged filter given by \eqref{Eq:solAvgZakai}.
\end{lemma}

\begin{proof}
Let us consider an independent copy of $(\widetilde{X}^{\delta},\widetilde U^\delta,\bar{\widetilde{U}})$, which has the same law as $(X^{\delta},U^\delta,\widetilde U)$, but which is independent of $(X^{\delta},U^\delta,\bar U,W)$. We have
\begin{align*}
\mathbb E_\theta^*\left(\phi_t^{\delta,\theta}[f]-\bar\phi_t^{\delta,\theta}[f]\right)^2
&\leq \cancelto{0}{\mathbb E_\theta^*\left(\phi_t^{\delta,\theta}[f]-\phi_t^{\delta,\theta}[\bar f_\theta]\right)^2}+\mathbb E_\theta^*\left(\phi_t^{\delta,\theta}[\bar f_\theta]-\bar\phi_t^{\delta,\theta}[f]\right)^2\
\end{align*}

The first term of the last display goes to zero by Lemma \ref{L:FilterConvergence3}. For the second term we have
\begin{align}
&\mathbb E_\theta^*\left(\phi_t^{\delta,\theta}[\bar f_\theta]-\bar\phi_t^{\delta,\theta}[\bar f_\theta]\right)^2\nonumber\\
&= \mathbb E_\theta^*\left[\mathbb E_\theta^*\left[\bar f_\theta(U_t^\delta)Z_t^{\delta,\theta}-\bar f_\theta(\bar U_t)\bar Z_t^{\delta,\theta}\Big|\mathcal Y_t^\delta\right]^2\right]\nonumber\\
&= \mathbb E_\theta^*\left[\mathbb E_\theta^*\left[(\bar f_\theta(U_t^\delta) Z_t^{\delta,\theta}-\bar f_\theta(\bar U_t) \bar Z_t^{\delta,\theta})(\bar f_\theta(\tilde U_t^\delta)\tilde Z_t^{\delta,\theta}-\bar f_\theta(\tilde{\bar U}_t)\tilde{\bar Z}_t^{\delta,\theta})\Big|\mathcal Y_t^\delta\right]\right]\nonumber\\
&= \mathbb E_\theta^*\left[\mathbb E_\theta^*\left[(\bar f_\theta(U_t^\delta) Z_t^{\delta,\theta}-\bar f_\theta(\bar U_t) \bar Z_t^{\delta,\theta})(\bar f_\theta(\tilde U_t^\delta)\tilde Z_t^{\delta,\theta}-\bar f_\theta(\tilde{\bar U}_t)\tilde{\bar Z}_t^{\delta,\theta})\Big|\mathcal F_t^{\delta,U,\tilde U,X,\tilde X}\right]\right]\nonumber\\
&= \mathbb E_\theta^*\left[\bar f_\theta(\tilde U_t^\delta)\left[\bar f_\theta(U_t^\delta)\exp\left(\int_0^t h(\tilde X_s^\delta,\tilde U_s^\delta)h(X_s^\delta, U_s^\delta)ds\right)-\bar f_\theta(\bar U_t)\exp\left(\int_0^th(\tilde X_s^\delta, \tilde U_s^\delta) \bar h_\theta(\bar U_s)ds\right)\right]\right]\nonumber\\
\nonumber
&\hspace{1cm}+\mathbb E_\theta^*\left[\bar f_\theta(\tilde {\bar U}_t)\left[\bar f_\theta(\bar U_t)\exp\left(\int_0^t\bar h_\theta(\bar U_s)\bar h_\theta(\bar{\widetilde{U}}_s)ds\right)-\bar f_\theta(U_t^\delta)\exp\left(\int_0^th(X_s^\delta, U_s^\delta) \bar h_\theta(\tilde{\bar U}_s)ds\right)\right]\right]\nonumber\\
&\rightarrow 0, \text{ as }\delta\downarrow 0\ .\label{Eq:AveragePhiConvergence}
\end{align}

The convergence in the last term is due to  Lemma \ref{L:NeededErgodicResult} proceeding as in the proof of Lemma \ref{L:FilterConvergence3}. Then, we have convergence in probability, $\mathbb P_\theta^*(|\phi_t^{\delta,\theta}[f]-\bar\phi_t^{\delta,\theta}[f]|\geq\epsilon)\leq\frac{1}{\eps^2}\mathbb E_\theta^*\left(\phi_t^{\delta,\theta}[f]-\bar\phi_t^{\delta,\theta}[f]\right)^2\rightarrow 0$ as $\delta\rightarrow 0$ for all $\epsilon>0$.

Next, we notice that for $r_{1},r_{2}>1$ such that $1/r_{1}+1/r_{2}=1$ and $p r_{2}\leq 2$
\begin{align}
&\mathbb E^{*}_{\theta}\left|\phi_t^{\delta,\theta}[f]/\phi_t^{\delta,\theta}[1]-\bar\phi_t^{\delta,\theta}[f]/\bar\phi_t^{\delta,\theta}[1]\right|^{p}=
\mathbb E^{*}_{\theta}\left|\frac{\phi_t^{\delta,\theta}[f]\bar\phi_t^{\delta,\theta}[1]-\bar\phi_t^{\delta,\theta}[f]\phi_t^{\delta,\theta}[1]}{\phi_t^{\delta,\theta}[1]\bar\phi_t^{\delta,\theta}[1]}\right|^{p}\nonumber\\
&\leq C \left(\mathbb E^{*}_{\theta}\left|\frac{1}{\phi_t^{\delta,\theta}[1]\bar\phi_t^{\delta,\theta}[1]}\right|^{pr_{1}}\right)^{1/r_{1}}
\left(\mathbb E^{*}_{\theta}\left|\phi_t^{\delta,\theta}[f]\bar\phi_t^{\delta,\theta}[1]-\bar\phi_t^{\delta,\theta}[f]\phi_t^{\delta,\theta}[1]\right|^{pr_{2}}\right)^{1/r_{2}}\nonumber\\
&\leq C \left(\mathbb E^{*}_{\theta}\left|\frac{1}{\phi_t^{\delta,\theta}[1]\bar\phi_t^{\delta,\theta}[1]}\right|^{pr_{1}}\right)^{1/r_{1}}
\left(\mathbb E^{*}_{\theta}\left|\phi_t^{\delta,\theta}[f]-\bar\phi_t^{\delta,\theta}[f]\right|^{pr_{2}}+\mathbb E^{*}_{\theta}\left|\phi_t^{\delta,\theta}[1]-\bar\phi_t^{\delta,\theta}[1]\right|^{pr_{2}}\right)^{1/r_{2}}\nonumber\\
&\leq C \left(\mathbb E^{*}_{\theta}\left|\frac{1}{\phi_t^{\delta,\theta}[1]}\right|^{2pr_{1}}+\mathbb E^{*}_{\theta}\left|\frac{1}{\bar\phi_t^{\delta,\theta}[1]}\right|^{2pr_{1}}\right)^{1/r_{1}}
\left(\mathbb E^{*}_{\theta}\left|\phi_t^{\delta,\theta}[f]-\bar\phi_t^{\delta,\theta}[f]\right|^{pr_{2}}+\mathbb E^{*}_{\theta}\left|\phi_t^{\delta,\theta}[1]-\bar\phi_t^{\delta,\theta}[1]\right|^{pr_{2}}\right)^{1/r_{2}}\nonumber
\end{align}
where boundedness of $f$ was used. By combining Lemma \ref{L:FilterConvergence3} and (\ref{Eq:AveragePhiConvergence})  we get that
\begin{align}
&\mathbb E^{*}_{\theta}\left|\phi_t^{\delta,\theta}[f]-\bar\phi_t^{\delta,\theta}[f]\right|^{pr_{2}}+\mathbb E^{*}_{\theta}\left|\phi_t^{\delta,\theta}[1]-\bar\phi_t^{\delta,\theta}[1]\right|^{pr_{2}}\rightarrow 0
\end{align}
In addition, we have
\begin{align}
\mathbb E^{*}_{\theta}\left|\frac{1}{\phi_t^{\delta,\theta}[1]}\right|^{2pr_{1}}&\leq
\mathbb E^{*}_{\theta}\left(Z_t^{\delta,\theta}\right)^{-2pr_{1}}=\mathbb E^{*}_{\theta}\left[E^{*}_{\theta}\left[\left(Z_t^{\delta,\theta}\right)^{-2pr_{1}}\right]|\mathcal F_t^{ U^{\delta}, X^{\delta}}\right]\nonumber\\
&=\mathbb E^{*}_{\theta}\left[e^{(2p^{2}r^{2}_{1}+pr_{1})\int_{0}^{t}|h_{\theta}(X^{\delta}_{s},U^{\delta}_{s})|^{2}ds}\right]\nonumber\\
&<\infty.
\end{align}
Similarly, we can also obtain $\mathbb E^{*}_{\theta}\left|\frac{1}{\bar\phi_t^{\delta,\theta}[1]}\right|^{2pr_{1}}<\infty$. Putting these  statements together we obtain that
\begin{align}
\lim_{\delta\downarrow 0}\mathbb E^{*}_{\theta}\left|\phi_t^{\delta,\theta}[f]/\phi_t^{\delta,\theta}[1]-\bar\phi_t^{\delta,\theta}[f]/\bar\phi_t^{\delta,\theta}[1]\right|^{p}&=0.\label{Eq:RatioPhis}
\end{align}

Then, by Cauchy-Schwartz inequality, we have
\begin{align}
&\mathbb E_{\alpha}\left|\pi_t^{\delta,\theta}[f]-\bar\pi_t^{\delta,\theta}[f]\right|=\mathbb E_{\alpha}\left|\phi_t^{\delta,\theta}[f]/\phi_t^{\delta,\theta}[1]-\bar\phi_t^{\delta,\theta}[f]/\bar\phi_t^{\delta,\theta}[1]\right|\nonumber\\
&\leq
 \left(\mathbb E^{*}_{\alpha}\left|Z^{\delta,\alpha}_{t}\right|^{q}\right)^{1/q}\left(\mathbb E^{*}_{\alpha}\left|\phi_t^{\delta,\theta}[f]/\phi_t^{\delta,\theta}[1]-\bar\phi_t^{\delta,\theta}[f]/\bar\phi_t^{\delta,\theta}[1]\right|^{p}\right)^{1/p}\nonumber\\
 &\leq
 \left(\sup_{\delta\in(0,1)}\left(\mathbb E^{*}_{\alpha}\left|Z^{\delta,\alpha}_{t}\right|^{q}\right)^{1/q}\right)\left(\mathbb E^{*}_{\theta}\left|\phi_t^{\delta,\theta}[f]/\phi_t^{\delta,\theta}[1]-\bar\phi_t^{\delta,\theta}[f]/\bar\phi_t^{\delta,\theta}[1]\right|^{p}\right)^{1/p}\nonumber\\
 &\rightarrow 0\ ,\nonumber
\end{align}
which goes to zero as $\delta\downarrow 0$ by Lemma \ref{L:BoundedZ} and by (\ref{Eq:RatioPhis}). 
The third line,  i.e.,
\[\mathbb E_{\alpha}^*\left|\phi_t^{\delta,\theta}[f]/\phi_t^{\delta,\theta}[1]-\bar\phi_t^{\delta,\theta}[f]/\bar\phi_t^{\delta,\theta}[1]\right|^{p}=\mathbb E_{\theta}^*\left|\phi_t^{\delta,\theta}[f]/\phi_t^{\delta,\theta}[1]-\bar\phi_t^{\delta,\theta}[f]/\bar\phi_t^{\delta,\theta}[1]\right|^{p}\ ,\]
follows because both $\phi_t^{\delta,\theta}$ and $\bar\phi_t^{\delta,\theta}$ are functionals of $Y^\delta$ (and no other random variable), and $Y^\delta$ is a Brownian motion under both measures $\mathbb P_{\alpha}^*$ and $\mathbb{P}_{\theta}^*$.

\end{proof}
Before moving on we should clarify the last step in the proof to Lemma \ref{L:FilterConvergence5}. It should be made clear that $Y^\delta$ is only observed to be $\mathbb P_\alpha^*$ Brownian motion when $\alpha\in\Theta$ denotes the parameter value under the measure $\mathbb P_\alpha$. In the proof we are always handling $Y^\delta$ underneath an unconditional expectation operator $\mathbb E_\theta^*$, which is actually an expectation conditional on the ground truth of the true parameter value being $\theta$. This notation can be made more explicitly in the following way: for any function $f$ of the path $Y^\delta$,
\[\mathbb E_\theta^* f(Y^\delta) =\mathbb E^*[f(Y^\delta)|\hbox{true parameter value}=\theta]= \mathbb Ef(W)\qquad\forall \theta\in\Theta\ ,\]
where $W$ is a Brownian motion. Hence, we are able to say
\[\mathbb E_\alpha^* f(Y^\delta)  = \mathbb E_\theta^* f(Y^\delta) \qquad\forall \alpha, \theta\in\Theta\ .\]

Now the proof of Theorem \ref{T:FilterConvergence1} follows:

\begin{proof}[Proof of Theorem \ref{T:FilterConvergence1}]
Let us prove just the second part of the theorem because the first part follows from a Chebyshev inequality and Lemma \ref{L:FilterConvergence3}. We prove it first for $f\in C_{b}(\mathcal{X}\times\mathcal{U})$. Then, we prove under the assumption that there exists $\eta>0$ such that $f\in\mathcal{A}_{\eta}^{\theta}$. So, let us assume that $f\in C_{b}(\mathcal{X}\times\mathcal{U})$. We start the proof by proving first that
\[\lim_{\delta\downarrow 0}\mathbb E_{\alpha}\left(\pi_t^{\delta,\theta}[f]-\bar \pi_t^{\delta, \theta}[f]\right)^2=0\ . \]
Letting $\bar f_\theta(u) = \int f(x,u)\mu_\theta(x,u)dx$, Lemma \ref{L:FilterConvergence3} implies convergence in $\mathbb P^*$ probability,
\[\mathbb P_{\theta}^*\left(\left|\phi_t^{\delta,\theta}[f]-\phi_t^{\delta,\theta}[\bar f_\theta]\right|>\eps\right)\leq \frac{1}{\eps}\mathbb E_{\theta}^*\left|\phi_t^{\delta,\theta}[f]-\phi_t^{\delta,\theta}[\bar f_\theta]\right|\rightarrow 0\qquad\forall\eps>0\ ,\]
Now consider any $f\in C_{b}(\mathcal{X}\times\mathcal{U})$, and again using a Cauchy-Schwartz inequality,
\begin{align}
\nonumber
&\lim_\delta\mathbb E_{\alpha}\left|\pi_t^{\delta,\theta}[f]-\bar{\pi}_t^{\delta,\theta}[f]\right|\\
\nonumber
&=\lim_\delta\mathbb E_{\alpha}\left|\pi_t^{\delta,\theta}[f]-\bar{\pi}_t^{\delta,\theta}[\bar f_\theta]\right|\\
\nonumber
&\leq \lim_\delta\mathbb E_{\alpha}\left|\pi_t^{\delta,\theta}[f]-\pi_t^{\delta,\theta}[\bar f_\theta]\right|+\cancelto{0}{\mathbb E_{\alpha}\left|\pi_t^{\delta,\theta}[\bar f_\theta]-\bar{\pi}_t^{\delta,\theta}[\bar f_\theta]\right|}\qquad\hbox{(by Lemma \ref{L:FilterConvergence5})}\\
\nonumber
&=\lim_\delta\mathbb E_{\alpha}^*Z_t^{\delta,\alpha}\left|\pi_t^{\delta,\theta}[f]-\pi_t^{\delta,\theta}[\bar f_\theta]\right|\\
\nonumber
&\leq \lim_\delta\left(\mathbb E_{\alpha}^*\left(Z_t^{\delta,\alpha}\right)^q\right)^{1/q}\left(\mathbb E_\alpha^*\left|\pi_t^{\delta,\theta}[f]-\pi_t^{\delta,\theta}[\bar f_\theta]\right|^p\right)^{1/p}\\
\nonumber
\nonumber
&\leq\underbrace{\sup_{\delta\in(0,1)}\left(\mathbb E_{\alpha}^*\left(Z_t^{\delta,\alpha}\right)^q\right)^{1/q}}_{<\infty}\lim_\delta\left(\mathbb E_\theta^*\left|\pi_t^{\delta,\theta}[f]-\pi_t^{\delta,\theta}[\bar f_\theta]\right|^p\right)^{1/p}\\
\label{eq:diff_converge}
&=0\ ,
\end{align}
where finiteness of $\sup_{\delta\in(0,1} \mathbb E_{\alpha}^*\left(Z_t^{\delta,\alpha}\right)^q$ follows from Lemma \ref{L:BoundedZ}, and $\lim_\delta\mathbb E_\theta^*\left|\pi_t^{\delta,\theta}[f]-\pi_t^{\delta,\theta}[\bar f_\theta]\right|^p=\lim_\delta\mathbb E_\theta^*\left|\phi_t^{\delta,\theta}[f]/\phi_t^{\delta,\theta}[1]-\phi_t^{\delta,\theta}[\bar f_\theta]/\phi_t^{\delta,\theta}[1]\right|^p=0$ follows as (\ref{Eq:RatioPhis}). 
This proves convergence in $L^1$, and convergence in $L^2$ follows from dominated convergence because the test function $f$ was assumed bounded so that
$\left|\pi_t^{\delta,\theta}[f]-\bar{\pi}_t^{\delta,\theta}[f]\right|^2\leq 2 \|f\|_\infty^2$. This completes the proof for $f\in C_{b}(\mathcal{X}\times\mathcal{U})$.

Let us complete the proof  by assuming that there exists an $\eta>0$ such that $f\in\mathcal{A}_{\eta}^\theta$. For $n\in\mathbb{N}$,   define
\[
u_{n}(x)=
\begin{cases}
x &,|x|\leq n\\
n \textrm{ sign}(x)&,|x|>n
\end{cases}
\]
and set $f_{n}(x)=u_{n}(f(x))$. Analogously define
\[
\pi^{\delta,\theta}_{t}[f_{n}]\doteq\mathbb E_{\theta}\left[f_{n}(X_t^\delta,U_t^\delta)\Big|\mathcal Y_t^\delta\right],\quad \bar{f}_{n,\theta}(u)=\int_{\mathcal{X}}f_{n}(x,u)\mu_{\theta}(x,u)dx.
\]
Since $f_{n}$ is bounded, we already know that $\lim_{\delta\downarrow 0}\mathbb E_{\alpha}\left(\pi_t^{\delta,\theta}[f_{n}]-\bar \pi_t^{\theta}[f_{n}]\right)^2=0$.
So, it is enough to prove that
\[
\lim_{n\rightarrow\infty}\limsup_{\delta\downarrow 0}\mathbb E_{\alpha}\left(\pi_t^{\delta,\theta}[f]-\pi_t^{\delta,\theta}[f_{n}]\right)^2=0
\]
and
\[
\lim_{n\rightarrow\infty}\limsup_{\delta\downarrow 0}\mathbb E_{\alpha}\left(\bar \pi_{t}^{\delta,\theta}[f]-\bar \pi_t^{\delta,\theta}[f_{n}]\right)^2=0.
\]
Both of these statements follow from the observation: for $\eta>0$ such that $f\in\mathcal{A}_{\eta}^\theta$ we have
\begin{align*}
|f(x,u)-f_{n}(x,u)|^{2+\eta/2}&\leq2^{2+\eta/2}|f(x,u)|^{2+\eta/2}\indicator{|f(x,u)|>n}\\
&\leq 2^{2+\eta}n^{-\eta/2}|f(x,u)|^{2+\eta}\ ,
\end{align*}
and in particular, letting $C=2^{2+\eta}$, $p=\sqrt{(2+\eta)/4}$ so that $2p^2 = 2+\eta/2$ and $\tfrac1q =1 - \sqrt{\tfrac{4}{2+\eta}}$, then taking the following similar set of steps as in equation \eqref{eq:diff_converge} we have
\begin{align}
\nonumber
&\lim_{n\rightarrow\infty}\limsup_{\delta\downarrow 0}\mathbb E_{\alpha}\left|\pi_t^{\delta,\theta}[f]-\pi_t^{\delta,\theta}[f_{n}]\right|^2\\
\nonumber
&\leq\lim_{n\rightarrow\infty}\limsup_{\delta\downarrow 0}\mathbb E_{\alpha}\mathbb E_\theta\left[\left|f(X_t^\delta,U_t^\delta)-f_n(X_t^\delta,U_t^\delta)\right|^2\Big|\mathcal Y_t^\delta\right]\\
\nonumber
&\leq\lim_{n\rightarrow\infty}\limsup_{\delta\downarrow 0}\mathbb E_{\alpha}^*Z_t^{\delta,\alpha}\mathbb E_\theta\left[\left|f(X_t^\delta,U_t^\delta)-f_n(X_t^\delta,U_t^\delta)\right|^2\Big|\mathcal Y_t^\delta\right]\\
\nonumber
&\leq\lim_{n\rightarrow\infty}\limsup_{\delta\downarrow 0}\left(\mathbb E_{\alpha}^*(Z_t^{\delta,\alpha})^q\right)^{1/q}\left(\mathbb E_\alpha^*\mathbb E_\theta\left[\left|f(X_t^\delta,U_t^\delta)-f_n(X_t^\delta,U_t^\delta)\right|^2\Big|\mathcal Y_t^\delta\right]^p\right)^{1/p}\\
\nonumber
&\leq\lim_{n\rightarrow\infty}\limsup_{\delta\downarrow 0}\left(\mathbb E_{\alpha}^*(Z_t^{\delta,\alpha})^q\right)^{1/q}\left(\mathbb E_\alpha^*\mathbb E_\theta\left[\left|f(X_t^\delta,U_t^\delta)-f_n(X_t^\delta,U_t^\delta)\right|^{2p}\Big|\mathcal Y_t^\delta\right]\right)^{1/p}\\
\nonumber
&=\lim_{n\rightarrow\infty}\limsup_{\delta\downarrow 0}\left(\mathbb E_{\alpha}^*(Z_t^{\delta,\alpha})^q\right)^{1/q}\left(\mathbb E_\theta^*\mathbb E_\theta\left[\left|f(X_t^\delta,U_t^\delta)-f_n(X_t^\delta,U_t^\delta)\right|^{2p}\Big|\mathcal Y_t^\delta\right]\right)^{1/p}\\
\nonumber
&=\lim_{n\rightarrow\infty}\limsup_{\delta\downarrow 0}\left(\mathbb E_{\alpha}^*(Z_t^{\delta,\alpha})^q\right)^{1/q}\left(\mathbb E_\theta\left[ (Z_t^{\delta,\theta})^{-1}\mathbb E_\theta\left[\left|f(X_t^\delta,U_t^\delta)-f_n(X_t^\delta,U_t^\delta)\right|^{2p}\Big|\mathcal Y_t^\delta\right]\right]\right)^{1/p}\\
&\leq C \lim_{n\rightarrow\infty}\limsup_{\delta\downarrow 0}n^{-\eta/2p^2}\left(\mathbb E_{\alpha}^*\left(Z_t^{\delta,\alpha}\right)^q\right)^{1/q}\left(\mathbb E_\theta (Z_t^{\delta,\theta})^{-q}\right)^{1/q}\left( \mathbb  E_{\theta}\left|f(X^{\delta}_{t},U_t^\delta)\right|^{2+\eta}\right)^{1/p^2}\nonumber\\
\nonumber
&=0\ .
\end{align}

The equality in the fifth line above, i.e.,
\[
\mathbb E_\alpha^*\mathbb E_\theta\left[\left|f(X_t^\delta,U_t^\delta)-f_n(X_t^\delta,U_t^\delta)\right|^{2p}\Big|\mathcal Y_t^\delta\right]=\mathbb E_\theta^*\mathbb E_\theta\left[\left|f(X_t^\delta,U_t^\delta)-f_n(X_t^\delta,U_t^\delta)\right|^{2p}\Big|\mathcal Y_t^\delta\right]
\]
follows because $\mathbb E_\theta\left[\left|f(X_t^\delta,U_t^\delta)-f_n(X_t^\delta,U_t^\delta)\right|^{2p}\Big|\mathcal Y_t^\delta\right]$ is a functional of $Y^\delta$ (and no other random variable), and $Y^\delta$ is a Brownian motion under both measures $\mathbb P_{\alpha}^*$ and $\mathbb{P}_{\theta}^*$.

The same limit can be shown for $\lim_{n\rightarrow\infty}\limsup_{\delta\downarrow 0}\mathbb E_{\alpha}\left(\bar \pi_{t}^{\delta,\theta}[f]-\bar \pi_t^{\delta,\theta}[f_{n}]\right)^2$, but with $\bar Z_t^{\delta,\theta}$ and $\bar Z_t^{\delta,\alpha}$ used instead. Due to ergodicity, the proof of
\[
\lim_{\delta\downarrow 0}\mathbb E_{\alpha}\left|\bar \pi_t^{\delta,\theta}[f]-\bar \pi_t^{\theta}[f]\right|=0\ ,
\]
follows similarly and thus omitted.  This concludes the proof of the theorem.
\end{proof}


\begin{thebibliography}{}

\bibitem[Bain and Crisan, 2009]{bainCrisan}
Bain, A. and Crisan, D. (2009).
\newblock {\em Fundamentals of Stochastic Filtering}.
\newblock Springer, New York, NY.

\bibitem[Bensoussan et~al., 1978]{BLP}
Bensoussan, A., Lions, J., and Papanicolaou, G. (1978).
\newblock {\em Asymptotic Analysis for Periodic Structures}, volume~5 of {\em
  Studies in Mathematics and its Applications}.
\newblock North-Holland Publishing Co., Amsterdam.

\bibitem[Billingsley, 1968]{Billingsley}
Billingsley, P. (1968).
\newblock {\em Convergence of Probability Measures}.
\newblock New York, J. Willey.

\bibitem[Elliott, 1993]{elliott93}
Elliott, R.J. (1993).
\newblock New Finite-Dimensional Filters and Smoothers for Noisily Observed Markov Chains.
\newblock {\em IEEE Transactions on Information Theory}, 39(1):265--271.


\bibitem[Fouque et~al., 2011]{FouqueEtall2011}
Fouque, J.-P., Papanicolaou, G., Sircar, R., and Solna, K. (2011).
\newblock Multiscale stochastic volatility for equity, interest rate, and
  credit derivatives, Cambridge University press, Cambridge, UK.

\bibitem[Givon et~al., 2009]{GivonStinisWeare2009}
Givon, G., Stinis, P.,  and Weare, J. (2009).
\newblock Variance reduction for particle filters of systems with time scale separation.
\newblock {\em IEEE Transactions on Signal Processing}, 57(2):424--435.

\bibitem[Hasminskii, 1980]{Hasminskii}
Hasminskii, R. (1980).
\newblock {\em Stochastic Stability of Differential Equations}.
\newblock Sijthoff and Noorhoff.

\bibitem[Imkeller et~al., 2013]{ImkellerSriPerkowskiYeong2012}
Imkeller, P., Namachchivaya, N.~S., Perkowski, N., and Yeong, H.~C. (2013).
\newblock Dimensional reduction in nonlinear filtering: a homogenization
  approach.
\newblock {\em Annals of Applied Probability}, 23(6):2290--2326.

\bibitem[Jirsa et~al., 2014]{Jisra2014}
Jirsa, V.~K., Stacey, W.~C., Quilichini, P.~P., Ivanova, A.~I., and Bernard, C.
  (2014).
\newblock On the nature of seizure dynamics.
\newblock {\em Brain}, 137(8):2210--2230.

\bibitem[Kushner, 1990]{Kushner}
Kushner, H.~J. (1990).
\newblock {\em Weak Convergence Methods and Singularly Perturbed Stochastic
  Control and Filtering Problems}.
\newblock Birkh\"{a}user, Boston-Basel-Berlin.

\bibitem[Kutoyants, 2004]{Kutoyants}
Kutoyants, Y. (2004).
\newblock {\em Statistical Inference for Ergodic Diffusion Processes}.
\newblock Springer, London.

\bibitem[Majda et~al., 2008]{Majda2008}
Majda, A.~J., Franzke, C., and Khouider, B. (2008).
\newblock An applied mathematics perspective on stochastic modelling for
  climate.
\newblock {\em Philosophical Transactions of the Royal Society A},
  366(1875):2429--2455.

\bibitem[Papanicolaou and Spiliopoulos, 2014]{PapanicolaouSpiliopoulos2014}
Papanicolaou, A. and Spiliopoulos, K. (2014).
\newblock Filtering the maximum likelihood for multiscale problems.
\newblock {\em Siam journal on Multiscale Modeling and Simulation},
  12(3):1193--1229.

\bibitem[Papavasiliou, 2007]{Papavasiliou2007}
Papavasiliou, A. (2007).
\newblock Particle flters for multiscale diffusions.
\newblock {\em ESAIM Proceedings}, 19:108--114.



\bibitem[Pardoux and Veretennikov, 2001]{PardouxVeretennikov1}
Pardoux, E. and Veretennikov, A. (2001).
\newblock On {P}oisson equation and diffusion approximation i.
\newblock {\em Annals of Probability}, 29(3):1061--1085.

\bibitem[Pardoux and Veretennikov, 2003]{PardouxVeretennikov2}
Pardoux, E. and Veretennikov, A. (2003).
\newblock On {P}oisson equation and diffusion approximation ii.
\newblock {\em Annals of Probability}, 31(3):1066--1092.

\bibitem[Park et~al., 2011]{ParkRozovskySowers2010}
Park, J., Rozovsky, B., and Sowers, R. (2011).
\newblock Efficient nonlinear filtering of a singularly perturbed stochastic
  hybrid system.
\newblock {\em LMS J. Computational Mathematics}, 14:254--270.

\bibitem[Park et~al., 2008]{ParkSriSowers2008}
Park, J., Sowers, R., and Namachchivaya, N.~S. (2008).
\newblock A problem in stochastic averaging of nonlinear filters.
\newblock {\em Stochastics and Dynamics}, 8:543--560.

\bibitem[Park et~al., 2010]{ParkSriSowers2011}
Park, J., Sowers, R., and Namachchivaya, N.~S. (2010).
\newblock Dimensional reduction in nonlinear filtering.
\newblock {\em Nonlinearity}, 23:305--324.

\bibitem[Rozovsky, 1991]{rozovsky1992}
Rozovsky, B. (1991).
\newblock A simple proof of uniqueness for {K}ushner and {Z}akai equations.
\newblock In Mayer-Wolf, E., editor, {\em Stochastic analysis}, pages 449--458.
  Boston: Academic Press.
\end{thebibliography}
\bibliographystyle{apalike}
\small{}

\end{document}